\documentclass[10pt]{amsart}
\usepackage{amsmath,amsthm}
\usepackage{amsfonts,amssymb}

\usepackage[english]{babel}
\usepackage{enumerate}

\usepackage{xcolor}

\usepackage[margin=3cm]{geometry}

\newtheorem{theorem}{Theorem}
\newtheorem{proposition}[theorem]{Proposition}
\newtheorem{lemma}[theorem]{Lemma}

\numberwithin{theorem}{section}

\theoremstyle{definition}
\newtheorem{definition}[theorem]{Definition}

\theoremstyle{remark}
\newtheorem{remark}[theorem]{Remark}

\newcommand{\mR}{\mathbb{R}}
\newcommand{\mC}{\mathbb{C}}
\newcommand{\mN}{\mathbb{N}}

\newcommand{\mS}{\mathbb{S}}

\newcommand{\cM}{\mathcal{M}}

\newcommand{\cP}{\mathcal{P}}
\newcommand{\cC}{\mathcal{C}}

\newcommand{\cH}{\mathcal{H}}

\newcommand{\ux}{\underline{x}}

\newcommand{\uy}{\underline{y}}
\newcommand{\uu}{\underline{u}}
\newcommand{\uv}{\underline{v}}
\newcommand{\uz}{\underline{z}}
\newcommand{\us}{\underline{s}}
\newcommand{\ut}{\underline{t}}
\newcommand{\uom}{\underline{\omega}}
\newcommand{\uta}{\underline{\tau}}
\newcommand{\utd}{\underline{\tau}^{\dagger}}
\newcommand{\uet}{\underline{\eta}}
\newcommand{\uetd}{\underline{\eta}^{\dagger}}

\newcommand{\upx}{\partial_{\underline{x}}}
\newcommand{\upy}{\partial_{\underline{y}}}
\newcommand{\upz}{\partial_{\underline{z}}}

\def\a{{\alpha}} 
\def\b{{\beta}}

\numberwithin{equation}{section}

\begin{document}
 
\title[The monogenic Hua-Radon transform and its inverse]{The monogenic Hua-Radon transform and its inverse}

\author{Denis Constales}
\address{Department of Electronics and Information Systems \\Faculty of Engineering and Architecture\\Ghent University\\Krijgslaan 281, 9000 Gent\\ Belgium.}
\email{Denis.Constales@UGent.be}

\author{Hendrik De Bie}
\address{Department of Electronics and Information Systems \\Faculty of Engineering and Architecture\\Ghent University\\Krijgslaan 281, 9000 Gent\\ Belgium.}
\email{Hendrik.DeBie@UGent.be}

\author{Teppo Mertens}
\address{Department of Electronics and Information Systems \\Faculty of Engineering and Architecture\\Ghent University\\Krijgslaan 281, 9000 Gent\\ Belgium.}
\email{Teppo.Mertens@UGent.be}

\author{Frank Sommen}
\address{Department of Electronics and Information Systems \\Faculty of Engineering and Architecture\\Ghent University\\Krijgslaan 281, 9000 Gent\\ Belgium.}
\email{Franciscus.Sommen@UGent.be}

\date{\today}

\begin{abstract}
The monogenic Hua-Radon transform is defined as an orthogonal projection on holomorphic functions in the Lie sphere. Extending the work of Sabadini and Sommen, \emph{J. Geom. Anal.}, {\bf{29}} (2019), 2709-2737, we determine its reproducing kernel. Integrating this kernel over the Stiefel manifold yields a linear combination of the zonal spherical monogenics. Using the reproducing properties of those monogenics we obtain an inversion for the monogenic Hua-Radon transform.
\end{abstract}

\maketitle

\tableofcontents

\section{Introduction}
\setcounter{equation}{0}

The Radon transform is a well-known integral transform with many applications in both pure and applied mathematics. It was extended to the Clifford setting by Sommen in \cite{Radon to Clifford, Radon and X-ray, Clifford and integral}, see also \cite{Quaternionic and Radon}.

The Szeg\H o-Radon transform, a variant of the Clifford Radon transform, was introduced in \cite{Szego} by Colombo, Sabadini and Sommen. It was originally defined as an orthogonal projection of a Hilbert module, but it can equally be written as an integral transform over the unit sphere $\mS^{m-1}$ with respect to a reproducing kernel. Like the Clifford Radon transform, it is a projection from a space of monogenic functions in $m$ variables over the real Clifford algebra $\mathbb{R}_m$ onto a space of monogenic functions in 2 variables over the complex Clifford algebra $\mathbb{C}_m$.

Applying the dual transform $\tilde{R}$ to the reproducing kernel of the Szeg\H o-Radon transform, i.e. taking the integral over a Stiefel manifold of the reproducing kernel, will result in a linear combination of the zonal spherical monogenics. Consequently, when acting on a monogenic polynomial $M_k(\ux)$ of degree $k$ with the composition of the Szeg\H o-Radon transform with its dual, one obtains a scalar multiple of $M_k(\ux)$, where the proportionality constant $\theta_k$ depends on $k$. As this scalar $\theta_k$ can be accounted for by the Gamma operator $\Gamma_{\ux}$, one can invert the Szeg\H o-Radon transform by applying the operator $\theta_{-\Gamma_{\ux}} \tilde{R}$. These ideas were established in \cite{Szego}.

It was shown in e.g. \cite{Hua, Mor,holomorphic} that monogenic functions on the unit ball admit a holomorphic extension in the Lie ball. Following this idea, Sabadini and Sommen defined several Radon-type transforms over the Lie sphere in \cite{radonlie}. They showed that each of these can be written as an integral transform over the Lie sphere with respect to a certain kernel. One of these is the monogenic Hua-Radon transform, which we will study in great detail.

In this paper we complete the work of Sabadini and Sommen, by determining explicitly the kernel of the monogenic Hua-Radon transform, which is defined as an orthogonal projection onto a space of holomorphic functions generated by a specific basis. Moreover, we will obtain an inversion for this transform, along the lines of \cite{Szego}. At this point we require that the dimension $m\geq 3$, which can be justified by the fact that most of the basis is annihilated in case $m<3$. It turns out that using the dual transform will not produce a full inversion for the monogenic Hua-Radon transform, but it will yield a projection operator of a holomorphic function onto the term in its Fischer decomposition corresponding to a certain power of $\uz$. Hence we will introduce a total monogenic Hua-Radon transform, for which the outlined process leads to a complete inversion.

The structure of the paper is as follows. In Section \ref{Section::preliminaries} we give the necessary preliminaries on Clifford algebras, the Lie sphere and Clifford analysis that will be needed in this paper. In Section \ref{Section::transform} we define the monogenic Hua-Radon transform as an orthogonal projection. Section \ref{Section::technical} contains some technical lemmas required in order to compute the kernel of the monogenic Hua-Radon transform in Section \ref{Section::kernel}, more precisely in Theorem \ref{Thm coeff mono hua}. The inversion of the monogenic Hua-Radon transform will be discussed in Section \ref{Section::inversion}. Finally, the more tedious calculations involving hypergeometric series and combinatorial identities will be performed in Appendix \ref{appendix A} and \ref{appendix B}.

\section{Preliminaries}\label{Section::preliminaries}
\setcounter{equation}{0}

In this section we introduce all notations and preliminary results that will be useful for the paper. We mostly follow the notations from \cite{radonlie}.

\subsection{Clifford algebras}

Consider the vector space $\mR^{m}$ with canonical orthonormal basis $( e_1, e_2, \ldots, e_{m})$. The Clifford algebra $\mR_{m}$, respectively $\mC_m$, is the real, respectively complex, algebra generated by the basis elements $e_1, e_2, \ldots, e_{m}$ under the relations
\begin{align*}
e_j^2 &= -1, \qquad j \in \{1, \ldots, m\},\\
e_j e_k + e_k e_j &=0, \phantom{-}\qquad j \neq k.
\end{align*}
Any element $\a$ in $\mR_m$ (or $\mC_m$) is of the form

\[
\a = \sum_{A\subset\{1,\ldots,m\}} \alpha_A e_A
\]
where $\a_A\in\mR$ (or $\a_A\in\mC$), $A = \{i_1,\ldots,i_l\}$ where $i_1<\ldots <i_l$ is a multi-index such that $e_A = e_{i_1}\ldots e_{i_l}$ and $e_{\emptyset} = 1$. The elements of $\mR_m$ (or $\mC_m$) which are linear combinations of only the basis vectors $e_j$ are called 1-vectors and will be underlined, e.g. $\uv = \sum_{j=1}^m v_j e_j$. The scalar part $\alpha_{\emptyset}$ of an element $\alpha$ will be denoted by $[\alpha]_0$.\\
The Hermitian conjugation is an automorphism on $\mC_m$ defined for $\a,\b\in\mC_m$ as

\begin{align*}
(\a\b)^{\dagger} &= \b^{\dagger}\a^{\dagger},\\
(\a + \b)^{\dagger} &= \a^{\dagger} + \b^{\dagger},\\
(\a_A e_A)^{\dagger} &= \overline{\a_A} e_A^{\dagger},\\
e_j^{\dagger}&=-e_j \qquad j\in\{1,\ldots,m\}.
\end{align*}
where $\overline{\a_A}$ stands for the complex conjugate of the complex number $\a_A$.

\subsection{Clifford analysis}

The norm of a 1-vector $\uv$ in either $\mR_m$ or $\mC_m$ is defined as

\[
|\uv|^2 = \sum_{j=1}^m v_j^2 = -\uv^2.
\]
In the real case, we can use $|\uv| = \sqrt{\sum_{j=1}^m v_j^2}$, whereas in the complex case we will be working with $|\uv|^2 = -\uv^2$, i.e. a complexified version of the square of the real norm.\\
The open unit ball with center at the origin in $\mR^m$ will be denoted by $B(0,1)$, while $\mS^{m-1}$ will denote the unit sphere, i.e. $\mS^{m-1} = \left\{\uv\in\mR^m \mid |\uv|^2 = 1\right\}.$
The area of the unit sphere is given by

\[
A_m = \frac{2\pi^{m/2}}{\Gamma\left(\frac{m}{2}\right)},
\]
where $\Gamma$ is the gamma function.\\
We define the scalar product of two 1-vectors $\uu$ and $\uv$ as $\langle \uu,\uv\rangle = \sum_{j=1}^m u_jv_j$. The wedge product of $\uu$ and $\uv$ is defined as $\uu\wedge\uv = \sum_{i<j}(u_iv_j-u_jv_i)e_{i}e_{j}$. Easy calculations show
\begin{equation}\label{product vectors}
\uu\phantom{.}\uv = -\langle\uu,\uv\rangle + \uu\wedge\uv.
\end{equation}

We will need the following result which was proven in \cite{Szego}.

\begin{lemma}\label{taulemma}
Let $\ut,\us\in\mS^{m-1}$ be such that $\langle\ut,\us\rangle = 0$ and let $\uta = \ut+i\us\in\mC^m$. Then $\utd = -\ut + i\us$ and
\begin{enumerate}[(i)]
\item
$\uta\phantom{.}\utd\uta = 4\uta$,
\item
$\uta^2 = (\utd)^2 = 0$,
\item
$\uta\phantom{.}\utd + \utd\uta = 4$.
\end{enumerate}
\end{lemma}
We will be working in a complex setting, hence we will use the complexified Dirac operator 

\[
\upz = \sum_{j=1}^m e_j\partial_{z_j}.
\]
Its square satisfies $\upz^2= - \Delta_{\uz}$, where $\Delta_{\uz} = \sum\limits_{j=1}^{m}\partial_{z_{j}}^2$ is the complexified Laplace operator. If we use the real Laplace operator, we will denote it by $\Delta_{\ux} = \sum\limits_{j=1}^{m}\partial_{x_{j}}^2$. The symbol of the Dirac operator $\upz$ is denoted by the vector variable
\[
\uz= \sum_{j=1}^{m}e_{j}z_j
\]
which is a complex-valued variable. If we are working with real variables, we will denote them by $\ux$. By complexifying the variables, the square of the norm and the Dirac operator, we can use the necessary results of real Clifford analysis.

\begin{definition}
A function $f:\Omega\subset \mC^m \to\mC_m$ which is continuously differentiable in the open set $\Omega$ is called (left) monogenic in $\Omega$ if $f$ is holomorphic and $f$ is in the kernel of the complexified Dirac operator $\upz$, i.e. $\upz f = \sum_{j=1}^m e_j \partial_{z_j}f = 0$. The right $\mC_m$-module of (left) monogenic functions in $\Omega$ is denoted by $\cM(\Omega)$.\\
A function $f:\Omega\subset \mC^m \to\mC_m$ which is continuously differentiable in the open set $\Omega$ is called harmonic in $\Omega$ if $f$ is holomorphic and $f$ is in the kernel of the complexified Laplace operator $\Delta_{\uz} = -\upz^2$. The right $\mC_m$-module of harmonic functions in $\Omega$ is denoted by $\cH(\Omega)$.
\end{definition}

Looking at the set of $k$-homogeneous polynomials $\cP_k(\Omega)$ with $\Omega\subset \mR^m$ we define the set of $k$-homogeneous monogenics and harmonics as
\begin{align*}
\cM_k(\Omega) &= \cM(\Omega)\cap\cP_k(\Omega),\\
\cH_k(\Omega) &= \cH(\Omega)\cap\cP_k(\Omega).
\end{align*}

It is a well known fact, (see e.g. \cite{Groen}), that for each $H_k(x)\in\cH_k(\Omega)$ there exist unique monogenic polynomials $M_k(x)\in\cM_k(\Omega),M_{k-1}(x)\in\cM_{k-1}(\Omega)$ such that
\begin{align*}
H_k(x) = M_k(x) + \ux M_{k-1}(x)
\end{align*}
Moreover these monogenics are determined by
\begin{align}
M_k(x) &= \left(1+\frac{1}{2k+m-2} \ux\upx\right)H_k(x), \label{monogenic projection harmonic}\\
M_{k-1}(x) &= -\frac{1}{2k+m-2} \upx. \nonumber
\end{align}

\subsection{The Lie ball}

As we are complexifying our variables, we will also complexify the space over which we are working. To this end we will consider the Lie sphere $LS^{m-1}$ and Lie ball $LB(0,1)$, instead of the unit sphere $\mS^{m-1}$ and unit ball $B(0,1)$.

\begin{definition}
The Lie ball can be defined as

\[
LB(0,1) = \{\uz=\ux+i\uy\in \mC^m \mid S_{\ux,\uy} \subset B(0,1)\}
\]
where $S_{\ux,\uy}$ is the codimension 2 sphere:

\[
S_{\ux,\uy} = \{\underline{u}\in\mR^m \mid \phantom{ } |\underline{u}-\ux| = |\uy|, \langle \underline{u} - \ux,\uy\rangle = 0\}.
\]
\end{definition}

\begin{remark}
Another way to introduce the Lie ball (see \cite{FourierBorel, Lie}) is to consider the Lie norm

\[
L(\uz)^2 = L(\ux+i\uy)^2 = \sup_{\underline{u}\in S_{\ux,\uy}} |\underline{u}|^2 = |\ux|^2 + |\uy|^2 + 2|\ux\wedge\uy|
\]
where $\uz = \ux+i\uy\in\mC^m$. Using the Lie norm we define the Lie ball as

\[
LB(0,1) = \{\uz\in\mC^m\mid L(\uz)<1\}.
\]
\end{remark}

\begin{definition}
The Lie sphere $LS^{m-1}$ is the set of points $\uz = \ux + i\uy\in\mC^m$ for which $S_{\ux,\uy} \subset \mS^{m-1}$ or equivalently

\[
LS^{m-1} = \{e^{i\theta} \uom \mid \uom\in\mS^{m-1}, \theta \in [0,\pi)\}.
\]
\end{definition}
Since monogenic functions $f(\ux)$ on $B(0,1)$ admit a holomorphic extension $f(\uz)$ in the Lie ball $LB(0,1)$ (see e.g. \cite{Hua, Mor,holomorphic}), we will be interested in the following space of holomorphic functions.

\begin{definition}
Let $\mathcal{OL}^2(LB(0,1))$ be the right $\mC_m$-module consisting of holomorphic functions $f:LB(0,1)\to\mC_m$ whose boundary value $f(e^{i\theta}\uom)$ belongs to $\mathcal{L}^2 (LS^{m-1})$, i.e. 

\[
\left[\int_{\mS^{m-1}} \int_0^\pi f^{\dagger}(e^{i\theta}\uom)f(e^{i\theta}\uom) d\theta dS(\uom)\right]_0<\infty.
\]
\end{definition}
We can equip $\mathcal{OL}^2(LB(0,1))$ with the following inner product 

\[
\langle f,g\rangle_{\mathcal{OL}^2(LB(0,1))} = \int_{\mS^{m-1}}\int_{0}^{\pi} f^{\dagger}(e^{i\theta}\uom)g(e^{i\theta}\uom)d\theta dS(\uom)
\]
For ease of notation we will denote this by $\langle \cdot,\cdot\rangle_{\mathcal{OL}^2}$.

\section{The monogenic Hua-Radon transform}\label{Section::transform}
\setcounter{equation}{0}
In this section, we will define the monogenic Hua-Radon transform and complete some proofs of Section 6 of \cite{radonlie}. Let us first remind the reader of the functions that we will be working with.\\
Let $\uta = \ut + i\us\in\mC_m$ with $\ut,\us\in\mS^{m-1}$ and $\ut\perp\us$, i.e. $\langle\ut,\us\rangle = 0$. We first define the following functions:

\begin{align*}
\psi_{\uta,2s,k}(z) &= \uta \langle\uz,\uta\rangle^{s+k}\langle\uz,\utd\rangle^{s}\\
\psi_{\uta,2s+1,k}(z) &= \utd\uta \langle\uz,\uta\rangle^{s+k+1}\langle\uz,\utd\rangle^{s}
\end{align*}
for $s,k\in\mN = \{0,1,2,\ldots\}$. Note that for each $\alpha,k\in\mN$ we have that $\psi_{\uta,\alpha,k}(z)$ is homogeneous of degree $\alpha+k$ in $\uz$.

The functions $\psi_{\uta,\alpha,k}(z)$ admit the following properties shown in \cite{radonlie}.

\begin{proposition}\label{recurrence psi}
The functions $\psi_{\uta,0,k}(\uz) = \uta \langle\uz,\uta\rangle^{k}$ are monogenic for each $k\in\mN$. Moreover for each $s,k\in\mN$ we have

\begin{align*}
\upz\psi_{\uta,2s+1,k}(z) &= 4(s+k+1) \psi_{\uta,2s,k}(z)\\
\upz\psi_{\uta,2s+2,k}(z) &= (s+1) \psi_{\uta,2s+1,k}(z).
\end{align*}
\end{proposition}
We can extend these functions to 

\begin{align*}
\psi_{\uta,2s,k}^j(z) &= \uz^j M[\psi_{\uta,2s,k}](z)\\
\psi_{\uta,2s+1,k}^j(z) &= \uz^j M[\psi_{\uta,2s+1,k}](z)
\end{align*}
where $j,s,k\in\mN$ and $M[\cdot]$ stands for the monogenic projection which can be defined as follows. For each $k$-homogeneous polynomial $P_k(z)$ there exist harmonic polynomials $H_{k-2\ell}(z)\in\cH(\Omega)$ of degree $k-2\ell$ for $\ell = 0,\ldots,\left \lfloor \frac{k}{2}\right \rfloor$ such that

\begin{equation}\label{fischer harmon}
P_k(z) = \sum_{\ell=0}^{\left \lfloor \frac{k}{2}\right \rfloor} (-\uz^2)^{\ell} H_{k-2\ell}(z),
\end{equation}
see e.g. \cite{Groen}. Using (\ref{monogenic projection harmonic}) we can refine (\ref{fischer harmon}) to

\begin{equation}\label{fischer mono}
P_k(z) = \sum_{j=0}^{k} \uz^{j} M_{k-j}(z).
\end{equation}
The monogenic projection is now defined as $M[P_k](\uz) = M_k(\uz)$.
\begin{definition}
The right $\mC_m$-module generated by $\{\psi_{\uta,\alpha,k}^j \mid k\in\mN\}$ will be denoted by $\mathfrak{M}^{j,\alpha}(\uta)$.
\end{definition}
\begin{definition}
The monogenic Hua-Radon transform $\cM_{\uta,j}$ is now defined as the orthogonal projection of $\mathcal{OL}^2(LB(0,1))$ onto the orthogonal direct sum $\bigoplus_{\alpha\in\mN}\mathfrak{M}^{j,\alpha}(\uta)$ (see \cite{radonlie}).
\end{definition}
The aim of Sections \ref{Section::transform} and \ref{Section::kernel} will be to write the monogenic Hua-Radon transform as an integral transform
\[
\cM_{\uta,j}[f](\uz) = \int_{0}^{\pi}\int_{\mS^{m-1}} K^j(\uz,e^{-i\theta} \uom) f(e^{-i\theta}\uom) dS(\uom)d\theta
\]
with respect to a certain reproducing kernel $K^j$ that is to be determined.

In order to determine the monogenic projection of $\psi_{\uta,\alpha,k}$, we will first compute its harmonic projection.
\begin{remark}\label{remark::dimension}
Note that if we are working in the 2-dimensional case $\psi_{\uta,\alpha,k}^j = 0$ for all $j,k$ and $\alpha\neq0,1$. Indeed if $m = 2$, then $\uta = e^{i\varphi} (e_1\pm ie_2)$ for some $\varphi\in [0,2\pi)$. This implies
\begin{align*}
\psi_{\uta,2s,k}(z) &= \uta e^{i\varphi k} (z_1 \pm iz_2)^{s+k} (-z_1\pm iz_2)^s = (-1)^s(-\uz^2)^{s} \uta  e^{i\varphi k} (z_1 \pm iz_2)^{k}, \\
\psi_{\uta,2s+1,k}(z) &= \utd\uta e^{i\varphi (k+1)} (z_1 \pm iz_2)^{s+k+1} (-z_1\pm iz_2)^s = (-1)^s(-\uz^2)^{s} \utd\uta  e^{i\varphi (k+1)} (z_1 \pm iz_2)^{k+1}.
\end{align*}
We can now see that $M[\psi_{\uta,2s,k}](\uz) = 0 = M[\psi_{\uta,2s+1,k}](\uz)$ whenever $s\neq 0$. Hence from now on we will assume $m\geq 3$.
\end{remark}
We can project $P_k$ in (\ref{fischer harmon}) onto each of its harmonic components $H_{k-2j}$ using the following projection operator (see e.g. \cite{projection} and \cite{Proj2}).

\begin{proposition}\label{Projection on harmonic}
The projection operator of a $k$-homogeneous polynomial onto its harmonic component of degree $k-2\ell$ is given by the following operator

\[
\sum_{j=0}^{\lfloor \frac{k}{2}\rfloor-\ell} \alpha_j (-\uz^2)^{j}\Delta_{\uz}^{j+\ell}
\]
where

\[
\alpha_j = \frac{(-1)^j(\frac{m}{2}+k-2\ell-1)}{4^{j+\ell}j!\ell!}\frac{\Gamma(\frac{m}{2}+k-2\ell-j-1)}{\Gamma(\frac{m}{2}+k-\ell)}
\]
\end{proposition}

The projection operator in Proposition \ref{Projection on harmonic} projects onto the harmonic components, but we need to project onto the monogenic components. Hence we will rewrite the previous operator using the Dirac operator and compose it with the projection onto its monogenic component (\ref{monogenic projection harmonic}):

\begin{proposition}\label{Projection on monogenic}
The projection operator of a $k$-homogeneous polynomial onto its monogenic component of degree $k-2l$ is given by the following operator

\[
\mathrm{proj}_{2\ell}^k =\sum_{j=0}^{k-2\ell} \beta_{j,2\ell} \uz^j \upz^{j+2\ell}
\]
where

\begin{align*}
\beta_{2j,2\ell} &= \left(\frac{-1}{4}\right)^{j+\ell}\frac{\Gamma(\frac{m}{2}+k-2\ell-j)}{j!\ell! \Gamma(\frac{m}{2}+k-\ell)},\\
\beta_{2j+1,2\ell} &= \left(\frac{-1}{4}\right)^{j+\ell}\frac{\Gamma(\frac{m}{2}+k-2\ell-j-1)}{2j!\ell! \Gamma(\frac{m}{2}+k-\ell)}.
\end{align*}
\end{proposition}
\begin{proof}
Using the projection operator of a harmonic polynomial onto its monogenic component, we get the following operator projecting a $k$-homogeneous polynomial onto its monogenic component of degree $k-2\ell$

\[
\mathrm{proj}_{2\ell}^k = \left(1+\frac{1}{2(k-2\ell)+m-2}\uz\upz\right)\sum_{j=0}^{\lfloor \frac{k}{2}\rfloor-\ell} \alpha_j (-1)^\ell  \uz^{2j}\upz^{2(j+\ell)}
\]
Now observe that for any function $f$ we have

\[
\upz\left[ \uz^{2j}f\right] = -2j\uz^{2j-1}f + \uz^{2j}\upz[f]
\]
as $\uz^{2j}$ is a scalar and hence will commute with all $e_j$ for $j=1,\ldots,m$. Consequently, we can rewrite the projection operator as follows:

\begin{align*}
\mathrm{proj}_{2\ell}^k =& \left(1+\frac{1}{2(k-2\ell)+m-2}\uz\upz\right)\sum_{j=0}^{\lfloor \frac{k}{2}\rfloor-\ell} \alpha_j (-1)^\ell  \uz^{2j}\upz^{2(j+\ell)} \\
=&\sum_{j=0}^{\lfloor \frac{k}{2}\rfloor-\ell} \alpha_j (-1)^l  \uz^{2j}\upz^{2(j+\ell)}+\frac{1}{2(k-2\ell)+m-2}\uz\sum_{j=0}^{\lfloor \frac{k}{2}\rfloor-\ell} \alpha_j (-1)^\ell  \upz\uz^{2j}\upx^{2(j+\ell)}\\
=&\sum_{j=0}^{k-2\ell} \beta_{j,2\ell} \uz^j \upz^{j+2\ell}
\end{align*}
with the coefficients $\beta_{j,2l}$ as in the formulation of the Proposition.
\end{proof}

We can now use Proposition \ref{Projection on monogenic} to calculate the monogenic projection of $\psi_{\uta, \alpha,k}$. The decomposition was already known in \cite{radonlie}, but the explicit expression for the constants was not. 

\begin{lemma}\label{Lemma Fisch decomp M[psi]}
For suitable constants $\mu_{l,\alpha,k}$, $l=1,\ldots,\alpha$ we have

\[
M[\psi_{\uta,\alpha,k}](z) = \sum_{j=0}^{\alpha}\mu_{j, \alpha,k} \uz^j\psi_{\uta,\alpha-j,k}(z).
\]
where

\begin{align*}
\mu_{2j,2s,k} &= (-1)^{j}\frac{\Gamma(\frac{m}{2}+2s+k-j)}{j! \Gamma(\frac{m}{2}+2s+k)}\frac{\Gamma(s+1)}{\Gamma(s-j+1)} \frac{\Gamma(s+k+1)}{\Gamma(s+k-j+1)} & &j=0,\ldots,s,\\
\mu_{2j+1,2s,k} &= (-1)^{j}\frac{\Gamma(\frac{m}{2}+2s+k-j-1)}{2j! \Gamma(\frac{m}{2}+2s+k)}\frac{\Gamma(s+1)}{\Gamma(s-j)} \frac{\Gamma(s+k+1)}{\Gamma(s+k-j+1)} & &j=0,\ldots,s-1,\\
\mu_{2j,2s+1,k} &= (-1)^{j}\frac{\Gamma(\frac{m}{2}+2s+k+1-j)}{j! \Gamma(\frac{m}{2}+2s+k+1)}\frac{\Gamma(s+1)}{\Gamma(s-j+1)} \frac{\Gamma(s+k+2)}{\Gamma(s+k-j+2)} & & j=0,\ldots,s,\\
\mu_{2j+1, 2s+1,k} &=2(-1)^{j}\frac{\Gamma(\frac{m}{2}+2s+k-j)}{j! \Gamma(\frac{m}{2}+2s+k+1)}\frac{\Gamma(s+1)}{\Gamma(s+1-j)} \frac{\Gamma(s+k+2)}{\Gamma(s+k-j+1)}  & &j=0,\ldots,s.
\end{align*}
\end{lemma}
\begin{proof}
Using Proposition \ref{Projection on monogenic} we get

\begin{align*}
M[\psi_{\uta,\alpha,k}](z) &= \text{proj}_{0}^{\alpha+k}(\psi_{\uta,\alpha,k})(z)\\
&=\sum_{j=0}^{\alpha + k} \beta_{j,0} \uz^j \upz^j \psi_{\uta,\alpha,k}(z)\\
\end{align*}
where 

\begin{align*}
\beta_{2j,0} &= \left(\frac{-1}{4}\right)^{j}\frac{\Gamma(\frac{m}{2}+\alpha+k-j)}{j! \Gamma(\frac{m}{2}+\alpha+k)}\\
\beta_{2j+1,0}&=\left(\frac{-1}{4}\right)^{j}\frac{\Gamma(\frac{m}{2}+\alpha+k-j-1)}{2j! \Gamma(\frac{m}{2}+\alpha+k)}.	
\end{align*}
If $\alpha=2s$ is even, then using Proposition \ref{recurrence psi} we get

\begin{align*}
\upz^{2j}\psi_{\uta,2s,k}(z) &= 4^j \frac{\Gamma(s+1)}{\Gamma(s-j+1)} \frac{\Gamma(s+k+1)}{\Gamma(s+k-j+1)}\psi_{\uta,2s-2j,k}(z)\\
\upz^{2j+1}\psi_{\uta,2s,k}(z) &= 4^j \frac{\Gamma(s+1)}{\Gamma(s-j)} \frac{\Gamma(s+k+1)}{\Gamma(s+k-j+1)}\psi_{\uta,2s-2j-1,k}(z)
\end{align*} 
and thus the sum reduces to

\[
M[\psi_{\uta,2s,k}](z) = \psi_{\uta,2s,k}(z) + \mu_{1,2s,l}\uz\psi_{\uta,2s-1,k}(z)+\ldots+\mu_{2s,2s,k}\uz^{2s}\psi_{\uta,0,k}(z)
\]
with
\begin{align*}
\mu_{2j, 2s, k} &= \beta_{2j,0}4^j \frac{\Gamma(s+1)}{\Gamma(s-j+1)} \frac{\Gamma(s+k+1)}{\Gamma(s+k-j+1)}\\
&=(-1)^{j}\frac{\Gamma(\frac{m}{2}+2s+k-j)}{j! \Gamma(\frac{m}{2}+2s+k)}\frac{\Gamma(s+1)}{\Gamma(s-j+1)} \frac{\Gamma(s+k+1)}{\Gamma(s+k-j+1)}\\
\mu_{2j+1, 2s, k} &= \beta_{2j+1,0}4^j \frac{\Gamma(s+1)}{\Gamma(s-j)} \frac{\Gamma(s+k+1)}{\Gamma(s+k-j+1)}\\
&=(-1)^{j}\frac{\Gamma(\frac{m}{2}+2s+k-j-1)}{2j! \Gamma(\frac{m}{2}+2s+k)}\frac{\Gamma(s+1)}{\Gamma(s-j)} \frac{\Gamma(s+k+1)}{\Gamma(s+k-j+1)}.
\end{align*}
The case $\alpha=2s+1$ is done in a similar way.
\end{proof}
\begin{remark}
Lemma \ref{Lemma Fisch decomp M[psi]} also shows that $M[\psi_{\uta,2s,k}] = 0 = M[\psi_{\uta,2s+1,k}]$ if $s\geq 1$ and $m=2$ just as we have shown in Remark \ref{remark::dimension}. This might not be obvious at first, but using Lemma \ref{Lemma Fisch decomp M[psi]} we have
\begin{align}
M[\psi_{\uta,2s,k}] &= \sum_{j=0}^{2s}\mu_{j, 2s,k} \uz^j\psi_{\uta,2s-j,k}(z)\nonumber\\
&= \sum_{j=0}^{s}\mu_{2j, 2s,k} \uz^{2j}\psi_{\uta,2s-2j,k}(z) + \sum_{j=0}^{s-1}\mu_{2j+1, 2s,k} \uz^{2j+1}\psi_{\uta,2s-2j-1,k}(z). \label{equa::m=2}
\end{align}
Using the calculation done in Remark \ref{remark::dimension}, we get
\begin{align*}
\uz^{2j}\psi_{\uta,2s-2j,k}(z) &= \uz^{2j}\uz^{2s-2j} \uta  e^{i\varphi k} (z_1 \pm iz_2)^{k}, \\
\uz^{2j+1}\psi_{\uta,2s-2j-1,k}(z) &= \uz^{2j+1} \uz^{2s-2j-2} \utd\uta  e^{i\varphi (k+1)} (z_1 \pm iz_2)^{k+1}\\
&= \uz^{2s-2} e^{i\varphi (k+1)} \uz(2 \mp 2ie _1 e_2)(z_1 \pm iz_2)^{k+1}\\
&= \uz^{2s-2} e^{i\varphi (k+1)}(2e_1(z_1^2+z_2^2)+2ie_2(z_1^2+z_2^2))(z_1 \pm iz_2)^{k}\\
&= \uz^{2s-2} e^{i\varphi k}\left(-2\uta\uz^2\right)(z_1 \pm iz_2)^{k}\\
&= -2\uz^{2s} \uta  e^{i\varphi k} (z_1 \pm iz_2)^{k}.
\end{align*}
Thus now we can rewrite (\ref{equa::m=2}) as a hypergeometric series:

\begin{align*}
(\ref{equa::m=2}) &= \left(\sum_{j=0}^{s}\mu_{2j, 2s,k} -2 \sum_{j=0}^{s-1}\mu_{2j+1, 2s,k}\right)\uz^{2s} \uta  e^{i\varphi k} (z_1 \pm iz_2)^{k}\\
&=\left({}_2 F_1\left([-s,-s-k],[-2s-k];1\right) \phantom{\frac{s}{2s+k}}\right.\\
&\left.- \frac{s}{2s+k}\phantom{.}_2 F_1\left([-s+1,-s-k],[-2s-k+1];1\right)\right)\uz^{2s} \uta  e^{i\varphi k} (z_1 \pm iz_2)^{k}
\end{align*}
where $\phantom{.}_2 F_1$ denotes the hypergeometric series given by
\[
\phantom{.}_2 F_1\left([a,b],[c];w\right) = \sum_{n=0}^{\infty} \frac{(a)_n(b)_n}{(c)_n}\frac{w^n}{n!}
\]
with $(a)_0 = 1$, $(a)_n = a(a+1)\ldots(a+n-1)$ the Pochhamer symbol.
Now using Chu-Vandermonde identity (see \cite{NIST}), we have
\begin{align*}
{}_2 F_1\left([-s,-s-k],[-2s-k];1\right) &= \dfrac{(-s)_s}{(-2s-k)_s} = \dfrac{s!(s+k)!}{(2s+k)!},\\
{}_2 F_1\left([-s+1,-s-k],[-2s-k+1];1\right) &= \dfrac{(-s+1)_{s-1}}{(-2s-k+1)_{s-1}} = \dfrac{(s-1)!(s+k)!}{(2s+k-1)!},
\end{align*}
which shows that $M[\psi_{\uta,2s,k}] = 0$. The case $\alpha = 2s+1$ is done analogously.
\end{remark}

\section{Technical lemmas}\label{Section::technical}
\setcounter{equation}{0}

In order to determine the kernel of the monogenic Hua-Radon transform in Section \ref{Section::kernel}, we will need the following technical lemmas. The first result was already proven in \cite{Szego}.
\begin{lemma}\label{lem4.7}
Let $\ut,\us\in\mR^m$ be such that $|\ut|=|\us|=1$ and $\langle\ut,\us\rangle = 0$, let $\uta = \ut+i\us\in\mC^m$ and $\uom\in\mS^{m-1}$. Then for $k\neq l$ we have

\[
\int_{\mS^{m-1}}\langle \uom,\utd\rangle^k\langle \uom,\uta\rangle^l dS(\uom)=0
\]

and

\[
\int_{\mS^{m-1}}\langle \uom,\utd\rangle^k\langle \uom,\uta\rangle^k dS(\uom)=(-1)^k 2\pi^{\frac{m}{2}}\frac{\Gamma(k+1)}{\Gamma(k+\frac{m}{2})}.
\]
\end{lemma}

We can write the integral in Lemma \ref{lem4.7} as the inner product $A_{m}\langle(-1)^k \langle\ux,\uta\rangle^k,\langle\ux,\uta\rangle^l\rangle_{\mS^{m-1}}$ with

\[
\langle P,Q\rangle_{\mS^{m-1}} = \frac{1}{A_{m}} \int_{\mS^{m-1}} \overline{P(\uom)}Q(\uom)dS(\uom)
\]
where $\overline{P(\uom)}$ is the complex conjugate of $P(\uom)$. It has been proven (see e.g. \cite{Xu}) that

\begin{equation}\label{Fisch Inprod}
2^k \frac{\Gamma\left(k+\frac{m}{2}\right)}{\Gamma\left(\frac{m}{2}\right)}\langle H_k,P_{\ell}\rangle_{\mS^{m-1}} = \left[\overline{H_k(\upx)}P_{\ell}(x)\right]_{x=0}
\end{equation}
where $H_k\in\cH_{k}$, $P_{\ell}$ is an $\ell$-homogeneous polynomial and $H_k(\upx)$ is the operator obtained by substituting $\partial_{x_i}$ for $x_i$ in $H(x)$.  Moreover if $k\neq \ell$ then

\begin{equation}\label{sphere inprod 0}
\langle H_k,H_\ell\rangle_{\mS^{m-1}} = 0
\end{equation}
where $H_k\in\cH_k, H_\ell\in\cH_\ell$. 

\begin{remark}\label{rem::sphere inprod 0 mono}
We can use (\ref{sphere inprod 0}) to get a similar result for monogenic polynomials, namely whenever $k\neq l$

\begin{align*}
\langle M_k,M_\ell\rangle_{\mS^{m-1}} = 0\\
\langle M_k,\ux M_\ell\rangle_{\mS^{m-1}} = 0\\
\langle \ux M_k,\ux M_\ell\rangle_{\mS^{m-1}} = 0,
\end{align*}
where $M_k \in \mathcal{M}_k$ and $M_\ell \in \mathcal{M}_\ell$.
\end{remark}
Now we can easily prove the following:

\begin{lemma}\label{lem7.7}
Let $\ut,\us\in\mathbb{R}^m$ be such that $|\ut|=|\us|=1$ and $\langle\ut,\us\rangle = 0$, let $\uta = \ut+i\us\in\mC^m$ and $\uom\in\mS^{m-1}$. Then for $k\neq \ell+1$ one has

\[
\int_{\mS^{m-1}}\langle \uom,\utd\rangle^k\langle \uom,\uta\rangle^{\ell}\uta\phantom{.}\uom dS(\uom)=0
\]
and

\[
\int_{\mS^{m-1}}\langle \uom,\utd\rangle^{\ell+1}\langle \uom,\uta\rangle^{\ell} \uta\phantom{.}\uom dS(\uom)=(-1)^{\ell} \pi^{\frac{m}{2}}\frac{\Gamma(\ell+2)}{\Gamma(\ell+1+\frac{m}{2})}\uta\phantom{.}\utd.
\]
\end{lemma}

\begin{proof}
We have 

\begin{align*}
\int_{\mS^{m-1}}\langle \uom,\utd\rangle^k\langle \uom,\uta\rangle^{\ell}\uta\phantom{.}\uom dS(\uom) &= A_{m}\langle (-1)^{\ell+1}\langle\ux,\utd\rangle^{\ell}\utd\ux,\langle\ux,\utd\rangle^k \rangle_{\mS^{m-1}}.
\end{align*}
Using (\ref{Fisch Inprod}), it is easy to see that

\begin{align*}
\langle (-1)^{\ell+1}\langle\ux,\utd\rangle^{\ell}\utd\ux,\langle\ux,\utd\rangle^k \rangle_{\mS^{m-1}} &= \frac{1}{2\left(k+\frac{m}{2}\right)} \langle (-1)^{\ell+1}\langle\ux,\utd\rangle^{\ell}\utd,\upx\langle\ux,\utd\rangle^k \rangle_{\mS^{m-1}}\\
&=\frac{1}{2\left(k+\frac{m}{2}\right)}\langle (-1)^{\ell+1}\langle\ux,\utd\rangle^{\ell}\utd,k\langle\ux,\utd\rangle^{k-1}\utd \rangle_{\mS^{m-1}}\\
&=\frac{1}{2\left(k+\frac{m}{2}\right)}\langle (-1)^{\ell}\langle\ux,\utd\rangle^{\ell},k\langle\ux,\utd\rangle^{k-1}\rangle_{\mS^{m-1}}\uta\phantom{.}\utd.
\end{align*}
The result now follows from Lemma \ref{lem4.7}.

\end{proof}
\begin{remark}\label{remark uometd}
Since $\utd$ and $\uta$ have the same properties as $-\ut\in\mS^{m-1}$ and $-\ut\perp\us$, we can easily replace $\uta$ by $\utd$. Moreover, we have

\begin{align*}
\int_{\mS^{m-1}}\langle \uom,\utd\rangle^k\langle \uom,\uta\rangle^\ell\uta\phantom{.}\uom dS(\uom) &= -\left(\int_{\mS^{m-1}} (-1)^{k+\ell}\langle \uom,\uta\rangle^k\langle \uom,\utd\rangle^\ell\uom\phantom{.}\utd dS(\uom)\right)^{\dagger}\\
&=\left\{\begin{array}{ll}
0 & \text{for }k \neq \ell+1,\\
(-1)^{\ell} \pi^{\frac{m}{2}}\frac{\Gamma(\ell+2)}{\Gamma(\ell+1+\frac{m}{2})}\uta\phantom{.}\utd & \text{for }k=\ell+1.
\end{array}\right.
\end{align*}

\end{remark}

In order to prove the following lemma, we will need to use Pizzetti's formula (see e.g. \cite{super, pizzetti}), which provides an easy way to calculate integrals over the unit sphere using the Laplacian. If $f$ is a polynomial, then

\begin{equation}\label{Pizetti}
\int_{\mS^{m-1}} f(\uom) \; {d S(\uom)} =  \sum_{k=0}^{\infty}  \frac{2 \pi^{m/2}}{4^{k} k!\Gamma(k+m/2)} (\Delta_{\ux}^k f )(0).
\end{equation}
This will allow us to show the following.

\begin{lemma}\label{lemma integral xtaux}
Let $\ut,\us\in\mR^m$ be such that $|\ut|=|\us|=1$ and $\langle\ut,\us\rangle = 0$, let $\uta = \ut+i\us\in\mC^m$ and $\uom\in\mS^{m-1}$. Then for $k\neq \ell$ we have

\[
\int_{\mS^{m-1}}\langle \uom,\utd\rangle^k\langle \uom,\uta\rangle^\ell\uom\phantom{.}\uta\phantom{.}\utd\uom dS(\uom)=0
\]
and

\begin{equation}\label{vgl:integral xtaux}
\int_{\mS^{m-1}}\langle \uom,\utd\rangle^{k}\langle \uom,\uta\rangle^k\uom\phantom{.}\uta\phantom{.}\utd\uom dS(\uom)=(-1)^{k} \pi^{\frac{m}{2}}\frac{k!}{\Gamma(k+1+\frac{m}{2})}(4 \uta\wedge\utd - m \uta \phantom{.}\utd-2k\utd\uta).
\end{equation}

\end{lemma}
\begin{proof}
To prove this, we will use Pizzetti's formula. It can be shown, using Lemma \ref{taulemma} and some straightforward calulations, that for $j\leq \min(k,\ell)$ one has

\begin{align*}
\Delta_{\ux}^{j} \left(\langle\ux,\utd\rangle^k\langle \ux,\uta\rangle^\ell\ux\phantom{.}\uta\phantom{.}\utd\ux\right) =& \Delta_{\ux}^j(\langle\ux,\utd\rangle^k\langle \ux,\uta\rangle^\ell)\ux\phantom{.}\uta\phantom{.}\utd\ux\\
&+ 8kj \Delta_{\ux}^{j-1}(\langle\ux,\utd\rangle^{k-1}\langle \ux,\uta\rangle^\ell)\utd\ux\\
&+ 8\ell j \Delta_{\ux}^{j-1}(\langle\ux,\utd\rangle^{k}\langle \ux,\uta\rangle^{\ell-1})\ux\phantom{.}\uta\\
&-4j(j-1)\Delta_{\ux}^{j-1}(\langle\ux,\utd\rangle^{k}\langle \ux,\uta\rangle^{\ell})\utd\uta\\
&+2j\Delta_{\ux}^{j-1}(\langle\ux,\utd\rangle^{k}\langle \ux,\uta\rangle^{\ell})T
\end{align*}
where $T = \sum_{i=1}^m e_i \uta\phantom{.}\utd e_i$. So now it is easy to see that if $k\neq l$, our integral will vanish, whereas if $k=\ell$, we have

\begin{align*}
\Delta_{\ux}^{k} \left(\langle\ux,\utd\rangle^k\langle \ux,\uta\rangle^k\ux\phantom{.}\uta\phantom{.}\utd\ux\right) =& \Delta_{\ux}^k(\langle\ux,\utd\rangle^k\langle \ux,\uta\rangle^k)\ux\phantom{.}\uta\phantom{.}\utd\ux\\
&+ 8k^2 \Delta_{\ux}^{k-1}(\langle\ux,\utd\rangle^{k-1}\langle \ux,\uta\rangle^k)\utd\ux\\
&+ 8k^2 \Delta_{\ux}^{k-1}(\langle\ux,\utd\rangle^{k}\langle \ux,\uta\rangle^{k-1})\ux\phantom{.}\uta\\
&-4k(k-1)\Delta_{\ux}^{k-1}(\langle\ux,\utd\rangle^{k}\langle \ux,\uta\rangle^{k})\utd\uta\\
&+2k\Delta_{\ux}^{k-1}(\langle\ux,\utd\rangle^{k}\langle \ux,\uta\rangle^{k})T\\
=&(-4)^k(k!)^2\ux\phantom{.}\uta\phantom{.}\utd\ux\\
&+8k^2(-4)^{k-1}(k-1)!k!\langle \ux,\uta\rangle \utd\ux\\
&+8k^2(-4)^{k-1}(k-1)!k!\langle \ux,\utd\rangle \ux\phantom{.}\uta\\
&-4k(k-1) (-4)^{k-1}(k!)^2\langle\ux,\utd\rangle\langle \ux,\uta\rangle\utd\uta\\
&+2k(-4)^{k-1}(k!)^2\langle\ux,\utd\rangle^{k}\langle \ux,\uta\rangle)T.
\end{align*}
Hence applying the Laplace operator one more time yields

\begin{align*}
\Delta_{\ux}^{k+1} \left(\langle\ux,\utd\rangle^k\langle \ux,\uta\rangle^k\ux\uta\utd\ux\right) =& 2(-4)^{k}k!(k + 1)!(T-2k\utd\uta).
\end{align*}
Hence after applying Pizzetti's formula we get

\begin{align*}
\int_{\mS^{m-1}}\langle \uom,\utd\rangle^{l+1}\langle \uom,\uta\rangle^l\uom\uta\utd\uom dS(\uom)&=\sum_{j=0}^{\infty} \frac{2\pi^{m/2}}{4^j j!\Gamma(j+m/2)} \Delta_{\ux}^j \left(\langle\ux,\utd\rangle^k\langle \ux,\uta\rangle^k\ux\uta\utd\ux\right)\vert_{\ux=0}\\
&=(-1)^{k} \pi^{\frac{m}{2}}\frac{k!}{\Gamma(k+1+\frac{m}{2})}(T-2k\utd\uta).
\end{align*}
We can now rewrite $T$ to get a neater result. Using (\ref{product vectors}) we get

\begin{align*}
T &= \sum_{i=1}^m e_i \uta\utd e_i = \sum_{i=1}^m e_i (-\langle\uta,\utd\rangle + \uta\wedge\utd) e_i\\
&= -\langle\uta,\utd\rangle(-m) + \sum_{i=1}^m e_i \uta\wedge\utd e_i = m\langle\uta,\utd\rangle + (4-m)\uta\wedge\utd\\
&= 4 \uta\wedge\utd - m(-\langle\uta,\utd\rangle + \uta\wedge\utd) = 4 \uta\wedge\utd - m \uta \phantom{.}\utd
\end{align*}
where we used Lemma 4.3 of \cite{Cnudde}
\end{proof}

\begin{remark}
As $\uta\wedge\utd$ is a bivector and $[\uta\phantom{.}\utd]_0 = [\utd\uta]_0 = 2$, we have that the scalar part of the right-hand side of (\ref{vgl:integral xtaux}) yields

\[
(-1)^{k+1}4 \pi^{\frac{m}{2}}\frac{k!}{\Gamma(k+\frac{m}{2})}.
\]
\end{remark}

\section{The kernel of the monogenic Hua-Radon transform}\label{Section::kernel}
\setcounter{equation}{0}

In order to determine the kernel of the monogenic Hua-Radon transform, we will need the following proposition.

\begin{proposition}\label{prop equalities psij inprod}
The following equalities hold:

\begin{equation}\label{inprod equalities}
\langle\psi_{\uta,\alpha,k}^j,\psi_{\uta,\alpha',k'}^j\rangle_{\mathcal{OL}^2} = \langle M[\psi_{\uta,\alpha,k}],M[\psi_{\uta,\alpha',k'}]\rangle_{\mathcal{OL}^2} = \langle M[\psi_{\uta,\alpha,k}],\psi_{\uta,\alpha',k'}\rangle_{\mathcal{OL}^2}.
\end{equation}
Moreover if $\alpha+k\neq\alpha'+k'$ then the three quantities in (\ref{inprod equalities}) will vanish.
\end{proposition}
\begin{proof}
Using the homogeneity of the functions we get

\begin{align*}
\langle\psi_{\uta,\alpha,k}^j,\psi_{\uta,\alpha',k'}^j\rangle_{\mathcal{OL}^2} =& \int_{0}^{\pi} e^{i\theta(\alpha'+k'-\alpha-k)}d\theta\\
&\times\int_{\mS^{m-1}}\left[M[\psi_{\uta,\alpha,k}](\uom)\right]^{\dagger} (-\uom^2)^j M[\psi_{\uta,\alpha',k'}](\uom)   dS(\uom)\\
=&\langle M[\psi_{\uta,\alpha,k}],M[\psi_{\uta,\alpha',k'}]\rangle_{\mathcal{OL}^2},
\end{align*}
where we used that $\uom^{\dagger} = -\uom$ and $-\uom^2 = 1$. We can now use (\ref{fischer mono}) to decompose $\psi_{\uta,\alpha',k'}$:

\[
\psi_{\uta,\alpha',k'}(z) = \sum_{\ell=0}^{\alpha'+k'} \uz^j M_{\alpha'+k'-\ell}(z)
\]
with $M_{\alpha'+k'-\ell}(z)$ a monogenic function for each $\ell$ and $M_{\alpha'+k'}(z) = M[\psi_{\uta,\alpha',k'}](z)$. So now we have

\begin{align*}
\langle M[\psi_{\uta,\alpha,k}],\psi_{\uta,\alpha',k'}\rangle_{\mathcal{OL}^2} =& \int_{0}^{\pi} e^{i\theta(\alpha'+k'-\alpha-k)}d\theta\\
&\times\int_{\mS^{m-1}}\left[M[\psi_{\uta,\alpha,k}](\uom)\right]^{\dagger}  \sum_{\ell=0}^{\alpha'+k'} \uom^{\ell} M_{\alpha'+k'-\ell}(\uom)  dS(\uom)\\
=& \sum_{\ell=0}^{\alpha'+k'}\int_{0}^{\pi} e^{i\theta(\alpha'+k'-\alpha-k)}d\theta\\
&\times\left\langle \overline{\left[M[\psi_{\uta,\alpha,k}](\ux)\right]^{\dagger}  \ux^{\ell}}, M_{\alpha'+k'-\ell}(\ux)\right\rangle_{\mS^{m-1}}.\\
\end{align*}

If we now use (\ref{Fisch Inprod}) we see that 

\begin{align*}
\left\langle \overline{\left[M[\psi_{\uta,\alpha,k}](\ux)\right]^{\dagger}  \ux^{\ell}}, M_{\alpha'+k'-\ell}(\ux)\right\rangle_{\mS^{m-1}} &= \left\langle \overline{\left[M[\psi_{\uta,\alpha,k}](\ux)\right]^{\dagger}  \ux^{\ell-1}}, \upx M_{\alpha'+k'-\ell}(\ux)\right\rangle_{\mS^{m-1}}\\
&= 0
\end{align*} for each $\ell\neq 0$, hence the last equality of (\ref{inprod equalities}) is proven. Moreover, if $\alpha+k\neq \alpha'+k'$ we have 
\[
\left\langle \overline{\left[M[\psi_{\uta,\alpha,k}](\ux)\right]^{\dagger}}, M_{\alpha'+k'-\ell}(\ux)\right\rangle_{\mS^{m-1}} = 0.
\]
Furthermore, $\upx M[\psi_{\uta,\alpha,k}](\ux) = 0$ and hence $\Delta_{\ux}M[\psi_{\uta,\alpha,k}](\ux) = 0$. But as $\Delta_{\ux}$ is a real-valued operator, we also have $\Delta_{\ux} \overline{\left[M[\psi_{\uta,\alpha,k}](\ux)\right]^{\dagger}} = 0$. So now using (\ref{sphere inprod 0}) we get the anticipated result.

\end{proof}

\begin{theorem}\label{Thm coeff mono hua}
The kernel of the monogenic Hua-Radon transform is given by

\[
K^j(\uz,e^{-i\theta}\uom) = \sum_{\alpha\in\mathbb{N}} K^{j,\alpha}(\uz,e^{-i\theta}\uom),
\]
where

\[
K^{j,\alpha}(\uz,e^{-i\theta}\uom) = \uz^j L^\alpha(\uz,e^{-i\theta}\uom) (e^{i\theta}\uom)^{-j}
\]
with

\[
L^\alpha(\uz,e^{-i\theta}\uom) = \sum_{k=0}^\infty \lambda_k^{\alpha} M[\psi_{\uta,\alpha,k}](\uz)\left[M[\psi_{\uta,\alpha,k}](e^{i\theta}\uom)\right]^{\dagger}
\]
and

\[
\lambda_k^{\alpha} = \left\{\begin{array}{ll}
\pi A_m \left(\frac{m}{2}-1\right)  \dfrac{\Gamma(\frac{m}{2}+2s+k)^2}{4s!(s+k)!\Gamma(s-1+\frac{m}{2})\Gamma(\frac{m}{2}+s+k)} & \alpha = 2s,\\
\pi A_m \left(\frac{m}{2}-1\right)\dfrac{\Gamma(\frac{m}{2}+2s+k+1)^2}{16s!(s+k+1)!\Gamma(s+\frac{m}{2})\Gamma(\frac{m}{2}+s+k)} & \alpha = 2s+1.
\end{array}\right.
\]
\end{theorem}
\begin{proof}
In order to prove this result, we will prove that $K^{j,\alpha}$ reproduces the basis elements $\psi^j_{\uta,\alpha,k}$ and hence it will reproduce each element of $\oplus_{\alpha\in\mN}\mathfrak{M}^{j,\alpha}$. Moreover, due to the way $K^{j,\alpha}$ is constructed, we get that the monogenic Hua-Radon transform is indeed represented by an integral transform with kernel $K^{j,\alpha}$. Thus we need to prove that

\begin{equation}\label{formula 12}
\psi_{\uta,\alpha,k}^j(z)=\int_{\mS^{m-1}}\int_0^\pi K^{j,\alpha}(z,e^{-i\theta}\omega)\psi_{\uta,\alpha,k}^j(e^{i\theta}\uom)dS(\uom)d\theta.
\end{equation}
The theorem will then follow using the orthogonality relations of $\psi^j_{\uta,\alpha,k}$, i.e. if $\alpha\neq\alpha'$ or $k\neq k'$ then $\langle\psi_{\uta,\alpha,k}^j,\psi_{\uta,\alpha',k'}^j\rangle_{\mathcal{OL}^2} = 0$, which was shown in \cite{radonlie}. We can rewrite equation (\ref{formula 12}) using the orthogonality relations of $M[\psi_{\uta,\alpha,k}]$ to the following identity

\[
\lambda_k^{\alpha}M[\psi_{\uta,\alpha,k}](\uz)\int_{\mS^{m-1}}\int_0^\pi M[\psi_{\uta,\alpha,k}]( e^{i\theta}\uom)^{\dagger}M[\psi_{\uta,\alpha,k}](e^{i\theta}\uom)dS(\uom)d\theta = M[\psi_{\uta,\alpha,k}](\uz).
\]
We will now show that the $\lambda_{\alpha}^k$ stated in the Theorem will be the required coefficient in order to reproduce $M[\psi_{\uta,\alpha,k}]$. We have, using Lemma \ref{Lemma Fisch decomp M[psi]}

\begin{align*}
\int_{\mS^{m-1}}\int_0^\pi &\left[M[\psi_{\uta,\alpha,k}](e^{i\theta}\uom)\right]^{\dagger}M[\psi_{\uta,\alpha,k}](e^{i\theta}\uom)dS(\uom)d\theta\\
&=\int_{\mathbb{S}^{m-1}}\int_0^\pi \left[M[\psi_{\uta,\alpha,k}]( e^{i\theta}\uom)\right]^{\dagger}\psi_{\uta,\alpha,k}(e^{i\theta}\uom)dS(\uom)d\theta\\
&=\sum_{j=0}^{\alpha}\mu_{j, \alpha,k} \int_{\mathbb{S}^{m-1}}\int_0^\pi \left [\psi_{\uta,\alpha-j,k}(e^{i\theta}\uom)\right]^{\dagger}(e^{i\theta}\uom)^{-j}\psi_{\uta,\alpha,k}(e^{i\theta}\uom)dS(\uom)d\theta.
\end{align*}
Let us now set

\[
\Phi_{j, \alpha,k} = \int_{\mathbb{S}^{m-1}}\int_0^\pi \left[\psi_{\uta,\alpha-j,k}(e^{i\theta}\uom)\right]^{\dagger}(e^{i\theta}\uom)^{-j}\psi_{\uta,\alpha,k}(e^{i\theta}\uom)dS(\uom)d\theta.
\]
Using the fact that $\psi_{\uta,\alpha,k}$ has degree of homogenicity $\alpha+k$, we get

\begin{align*}
\Phi_{j,\alpha,k} &= \int_{\mathbb{S}^{m-1}}\int_0^\pi e^{-i\theta(\alpha-j+k)}(\psi_{\uta,\alpha-j,k}(\uom))^{\dagger}(e^{i\theta}\uom)^{-j}e^{i\theta(\alpha+k)}\psi_{\uta,\alpha,k}(\uom)dS(\uom)d\theta\\
&=\pi\int_{\mathbb{S}^{m-1}}(\psi_{\uta,\alpha-j,k}(\uom))^{\dagger}(\uom)^{-j}\psi_{\uta,\alpha,k}(\uom)dS(\uom).
\end{align*}
First we consider the case $\alpha=2s$ even. For $j=2l$ we have

\begin{align*}
\Phi_{2l, 2s, k} &= (-1)^{-l}\pi\int_{\mathbb{S}^{m-1}}(\psi_{\uta,\alpha-j,k}(\uom))^{\dagger}\psi_{\uta,\alpha,k}(\uom)dS(\uom)\\
&=(-1)^{k-l}\pi\utd\uta\int_{\mathbb{S}^{m-1}}\langle \uom,\utd\rangle^{k-l+2s}\langle \uom,\uta\rangle^{k-l+2s} dS(\uom).
\end{align*}
Hence by Lemma \ref{lem4.7}, we get

\begin{align*}
\Phi_{2l, 2s, k} &= (-1)^{k-l} \pi\utd\uta (-1)^{k-l+2s}2\pi^{\frac{m}{2}}\frac{\Gamma(k-l+2s+1)}{\Gamma(\frac{m}{2}+k-l+2s)}\\
&=2\pi^{\frac{m}{2}+1}\utd\uta \frac{\Gamma(k-l+2s+1)}{\Gamma(\frac{m}{2}+k-l+2s)}.
\end{align*}
Now defining
\[
\phi_{2l,2s,k} = 2\pi^{\frac{m}{2}+1}\frac{\Gamma(k-l+2s+1)}{\Gamma(\frac{m}{2}+k-l+2s)},
\]
we have $\Phi_{2l, 2s, k} =\phi_{2l,2s,k} \utd\uta$.\\
For $j=2l+1$, we have by Lemma \ref{lem7.7}

\begin{align*}
\Phi_{2l+1, 2s, k} &= (-1)^{-l+1}\pi\int_{\mathbb{S}^{m-1}}(\psi_{\uta,\alpha-j,k}(\uom))^{\dagger}\uom\psi_{\uta,\alpha,k}(\uom)dS(\uom)\\
&=(-1)^{k+l}\pi\int_{\mathbb{S}^{m-1}}\langle \uom,\utd\rangle^{2s-l+k}\langle \uom,\uta\rangle^{2s-l+k-1} \utd\uta\phantom{.}\uom\phantom{.}\uta dS(\uom)\\
&= -4\pi^{\frac{m}{2}+1} \frac{\Gamma(2s-l+k+1)}{\Gamma(2s-l+k+\frac{m}{2})} \utd\uta.
\end{align*}
Write
\[
\phi_{2l+1,2s,k} = -4\pi^{\frac{m}{2}+1} \frac{\Gamma(2s-l+k+1)}{\Gamma(2s-l+k+\frac{m}{2})},
\]
then we have that $\Phi_{2l+1, 2s, k} = \phi_{2l+1,2s,k} \utd\uta$. Note that $\Phi_{j,2s,k}$ is a multiple of $\utd\uta$ for each $j$. Hence we have
\begin{equation}
\lambda_k^{\alpha}M[\psi_{\uta,\alpha,k}](\uz)\left(\sum_{l=0}^s \mu_{2l, 2s, k}\phi_{2l,2s,k} + \sum_{l=0}^{s-1} \mu_{2l+1, 2s, k} \phi_{2l+1,2s,k}\right)\utd\uta = M[\psi_{\uta,\alpha,k}](\uz).
\end{equation}
Now note that each of the terms in the decomposition of $M[\psi_{\uta,2s,k}]$ shown in Lemma \ref{Lemma Fisch decomp M[psi]} is a multiple of $\uta$ since
\begin{align*}
\psi_{\uta,2s-2j,k}(\uz) &= \langle\uz,\uta\rangle^{s-j+k}\langle\uz,\utd\rangle^{s-j} \uta,\\
\psi_{\uta,2s-2j-1,k}(\uz) &= \langle\uz,\uta\rangle^{s-j+k-1}\langle\uz,\utd\rangle^{s-j} \utd\uta.
\end{align*}
Thus if we use Lemma \ref{taulemma} ($i$), i.e. $\uta\utd\uta = 4\uta$, we have
\[
4\lambda_k^{2s}\left(\sum_{l=0}^s \mu_{2l, 2s, k}\phi_{2l,2s,k} + \sum_{l=0}^{s-1} \mu_{2l+1, 2s, k} \phi_{2l+1,2s,k}\right) = 1
\]
Summarizing the case $\alpha = 2s$ even, we have

\begin{align*}
\frac{1}{4}(\lambda_k^{2s})^{-1} =& \sum_{l=0}^s \mu_{2l, 2s, k}\phi_{2l,2s,k} + \sum_{l=0}^{s-1} \mu_{2l+1, 2s, k} \phi_{2l+1,2s,k}\\
=&\phantom{.}\frac{2\pi^{\frac{m}{2}+1}\Gamma(s+1)\Gamma(s+k+1)}{\Gamma(\frac{m}{2}+2s+k)}\left(\sum_{l=0}^s  (-1)^l\frac{\Gamma(2s+k-l+1)}{\Gamma(l+1)\Gamma(s-l+1)\Gamma(s+k-l+1)}\right.\\
&-\left.\sum_{l=0}^{s-1}  (-1)^l\frac{\Gamma(2s+k-l+1)}{\Gamma(l+1)\Gamma(s-l)\Gamma(s+k-l+1)(\frac{m}{2}+2s+k-l-1)}\right).
\end{align*}
Now consider the case where $\alpha=2s+1$ is odd. For $j=2l$ we have

\begin{align*}
\Phi_{2l, 2s+1, k} &= (-1)^{-l}\pi\int_{\mathbb{S}^{m-1}}(\psi_{\uta,\alpha-j,k}(\uom))^{\dagger}\psi_{\uta,\alpha,k}(\uom)dS(\uom)\\
&=(-1)^{k-l+1}\pi 4 \utd\uta\int_{\mathbb{S}^{m-1}}\langle \uom,\utd\rangle^{2s+k-l+1}\langle \uom,\uta\rangle^{2s+k-l+1} dS(\uom).
\end{align*}
Hence by Lemma \ref{lem4.7}, we get

\begin{align*}
\Phi_{2l, 2s+1, k} &= (-1)^{k-l+1} \pi 4 \utd\uta (-1)^{2s+k-l+1}2\pi^{\frac{m}{2}}\frac{\Gamma(2s+k-l+2)}{\Gamma(\frac{m}{2}+2s+k-l+1)}\\
&=8\pi^{\frac{m}{2}+1}\utd\uta \frac{\Gamma(k-l+2s+2)}{\Gamma(\frac{m}{2}+k-l+2s+1)}.
\end{align*}
If we now define
\[
\phi_{2l,2s+1,k} = 8\pi^{\frac{m}{2}+1} \frac{\Gamma(k-l+2s+2)}{\Gamma(\frac{m}{2}+k-l+2s+1)},
\]
we have $\Phi_{2l, 2s+1, k}=\phi_{2l,2s+1,k}\utd\uta$.\\
For $j=2l+1$, we have by Lemma \ref{lem7.7}

\begin{align*}
\Phi_{2l+1, 2s+1, k} &= (-1)^{-l+1}\pi\int_{\mathbb{S}^{m-1}}(\psi_{\uta,\alpha-j,k}(\uom))^{\dagger}\uom\psi_{\uta,\alpha,k}(\uom)dS(\uom)\\
&=(-1)^{k+l+1}\pi\int_{\mathbb{S}^{m-1}}\langle \uom,\utd\rangle^{2s-l+k}\langle \uom,\uta\rangle^{2s-l+k+1} \utd\uom\phantom{.}\utd\uta dS(\uom)\\
&= -4\pi^{\frac{m}{2}+1} \frac{\Gamma(2s-l+k+2)}{\Gamma(2s-l+k+1+\frac{m}{2})} \utd\uta.
\end{align*}
We can now define
\[
\phi_{2l+1,2s+1,k} = -4\pi^{\frac{m}{2}+1} \frac{\Gamma(2s-l+k+2)}{\Gamma(2s-l+k+1+\frac{m}{2})},
\]
so that $\Phi_{2l+1, 2s+1, k}=\phi_{2l+1,2s+1,k}\utd\uta$.\\
Summarizing the case $\alpha = 2s+1$ odd, we have

\begin{align*}
\frac{1}{4}(\lambda_k^{2s+1})^{-1} =& \sum_{l=0}^s \mu_{2l, 2s+1, k}\phi_{2l,2s+1,k} + \sum_{l=0}^{s} \mu_{2l+1, 2s+1, k} \phi_{2l+1,2s+1,k}\\
=&\phantom{.}\frac{8\pi^{\frac{m}{2}+1}\Gamma(s+1)\Gamma(s+k+2)}{\Gamma(\frac{m}{2}+2s+k+1)}\left(\sum_{l=0}^s  (-1)^l\frac{\Gamma(2s+k-l+2)}{\Gamma(l+1)\Gamma(s-l+1)\Gamma(s+k-l+2)}\right.\\
&-\left.\sum_{l=0}^{s}  (-1)^l\frac{\Gamma(2s+k-l+2)\Gamma(\frac{m}{2}+2s+k-l)}{\Gamma(l+1)\Gamma(s-l+1)\Gamma(s+k-l+1)\Gamma(\frac{m}{2}+2s+k-l+1)}\right).
\end{align*}
After some lengthy and straightforward calculations which are contained in Appendix \ref{appendix A}, we get the result stated in the theorem.
\end{proof}

\section{Inversion of the monogenic Hua-Radon transform}\label{Section::inversion}
\setcounter{equation}{0}

Recall that the zonal spherical monogenics $\cC_k(\ux,\uy)$ can be written in terms of Gegenbauer polynomials $C_{k}^{\frac{m}{2}-1}$ and $C_{k-1}^{\frac{m}{2}-1}$ (see \cite{planewaves})

\begin{equation}\label{zonal gegenbauer identity}
\cC_k(\ux,\uy) = \frac{(|\ux| |\uy|)^k}{m-2}\left((k+m-2) C_k^{\frac{m}{2}-1}(t) + (m-2) \frac{\ux\wedge\uy}{|\ux| |\uy|} C_{k-1}^{\frac{m}{2}-1}(t)\right)
\end{equation}
with $t= \langle\ux,\uy\rangle/(|\ux||\uy|)$. Moreover they exhibit the following properties:

\begin{enumerate}[(i)]
\item 
$\cC_k(\ux,\uy)$ is a homogeneous polynomial of degree $k$ in $\ux$ and $\uy$,

\item
$\upx \cC_k(\ux,\uy) = 0$ and $\cC_k(\ux,\uy)\upy = 0$, where $\cC_k(\ux,\uy)\upy = \sum_{j=1}^m (\partial_{y_j}\cC_k(\ux,\uy))e_j$,  

\item
For every $\sigma\in\mbox{Spin}(m) = \{\prod_{i=1}^{2r} \sigma_i\vert \sigma_i \in\mS^{m-1}\}$ the functions $\cC_k(\ux,\uy)$ have the invariance property

\[
\sigma \cC_k(\overline{\sigma}\ux\sigma,\overline{\sigma}\uy \sigma) \overline{\sigma} = \cC_k(\ux,\uy).
\]
where we note that $\mbox{Spin}(m)/\{-1,1\} \simeq SO(m)$, see \cite{Friederich spin, Gil Mur spin}.
\end{enumerate}
Recall the following lemma proven in \cite{Szego} which proves that the zonal spherical monogenics $\cC_k(\ux,\uy)$ are, up to a constant, the unique Clifford algebra valued functions which admit these properties.

\begin{lemma}\label{Unicity zonal mono}
Let $F(\ux,\uy)$ be a Clifford valued function satisfying the properties $(i)-(ii)-(iii)$ above. Then there exist complex constants $\lambda$ and $\mu$ such that

\[
F(\ux,\uy) = (\lambda + \mu e_{1\ldots m})\cC_{k}(\ux,\uy).
\]
\end{lemma}

The functions $M[\psi_{\uta,\alpha,k}]$ that we defined earlier in Section \ref{Section::transform}, meet ($i$) and ($ii$). But they do not satisfy property ($iii$). However they do admit the following Lemma.

\begin{lemma}\label{Lemma spin on M[psi]}
Let $\sigma\in\text{Spin}(m)$, then it holds that

\[
\sigma M[\psi_{\uta,\alpha,k}](\overline{\sigma}\ux\sigma)\overline{\sigma} = M[\psi_{\sigma\uta\overline{\sigma},\alpha,k}](x).
\]
\end{lemma}
\begin{proof}
Using Lemma \ref{Lemma Fisch decomp M[psi]} we have

\[
M[\psi_{\uta,\alpha,k}](x) = \psi_{\uta,\alpha,k}(x) + \mu_1\ux\psi_{\uta,\alpha-1,k}(x)+\ldots+\mu_{\alpha}\ux^{\alpha}\psi_{\uta,0,k}(x).
\]
Hence it suffices to look at $\sigma\psi_{\uta,\alpha,k}(\overline{\sigma} \ux\sigma)\overline{\sigma}$. Now observing that $\sigma\overline{\sigma} = \overline{\sigma}\sigma = 1$ and that $\langle \overline{\sigma} \ux\sigma,\uta\rangle = \langle\ux, \sigma\uta\overline{\sigma}\rangle$, we get 

\[
\sigma\psi_{\uta,\alpha,k}(\overline{\sigma} \ux\sigma)\overline{\sigma} = \psi_{\sigma\uta\overline{\sigma},\alpha,k}(x)
\]
and thus the statement is proven.
\end{proof}
We now have the following.

\begin{proposition}\label{zonal equality stiefel integral}
For $\alpha,k\in\mN$ we have

\begin{equation}\label{integral over stiefel}
\int_{\mS^{m-1}}\int_{\mS^{m-2}} M[\psi_{\uta,\alpha,k}](x)M[\psi_{\uta,\alpha,k}](y)^{\dagger} dS(\us)dS(\ut) = \gamma_{\alpha,k} \cC_{\alpha + k}(x,y)
\end{equation}
for a suitable real or complex constant $\gamma_{n,k}$ provided in Proposition \ref{prop comp gam}.
\end{proposition}
\begin{proof}
If we define the integral on the left of (\ref{integral over stiefel}) as $L_{\alpha,k}(x,y)$, then we see that this integral is a homogeneous polynomial of degree $\alpha+k$ in both $\ux$ and $\uy$. Moreover it is left monogenic in $\ux$ and right monogenic in $\uy$.\\
Furthermore, let $\sigma\in \text{Spin}(m)$, then by using Lemma \ref{Lemma spin on M[psi]} we get

\begin{align*}
\sigma L_{k}(\overline{\sigma}\ux\sigma, \overline{\sigma}\uy\sigma)\overline{\sigma} &= \int_{\mS^{m-1}}\int_{\mS^{m-2}} \sigma M[\psi_{\uta,\alpha,k}](\overline{\sigma}\ux\sigma)M[\psi_{\uta,\alpha,k}](\overline{\sigma}\uy\sigma)^{\dagger}\overline{\sigma} dS(\us)dS(\ut)\\
&= \int_{\mS^{m-1}}\int_{\mS^{m-2}}  M[\psi_{\sigma\uta\overline{\sigma},\alpha,k}](x)M[\psi_{\sigma\uta\overline{\sigma},\alpha,k}](y)^{\dagger} dS(\us)dS(\ut)\\
&= \int_{\mS^{m-1}}\int_{\mS^{m-2}} M[\psi_{\uta,\alpha,k}](x)M[\psi_{\uta,\alpha,k}](y)^{\dagger} dS(\us)dS(\ut).
\end{align*}

Hence $L_{\alpha,k}(x,y)$ is Spin-invariant, so all the conditions of Lemma \ref{Unicity zonal mono} are fulfilled. Now noting that $L_{\alpha,k}(x,x)$ is the sum of a scalar $a_{\alpha,k}(\ux)$ and a bivector $b_{\alpha,k}(\ux)$, $\cC_{\alpha+k}(\ux,\ux)$ is a scalar, we get 
\[
a_{\alpha,k}(\ux) + b_{\alpha,k}(\ux) = (\lambda + \mu e_{1\ldots m})\cC_{\alpha+k}(\ux,\ux).
\]
If we use the fact that $m\geq 3$ (see Remark \ref{remark::dimension}), we see that $b_{\alpha,k}(\ux) = 0$ and $\mu = 0$. Hence we have

\[
L_{\alpha,k}(x,y) = \gamma_{\alpha,k} \cC_{\alpha+k}(x,y).
\]

\end{proof}

It is possible to compute the constants $\gamma_{\alpha,k}$ in Proposition \ref{zonal equality stiefel integral} explicitly. They are as follows:

\begin{proposition}\label{prop comp gam}
For $s,k\in\mN$ we have

\begin{align*}
\gamma_{2s,k} &= \frac{s!(s+k)!\Gamma\left(\frac{m}{2}+s-1\right)}{\Gamma\left(\frac{m}{2}+2s+k\right)^2}\Gamma\left(\frac{m}{2}+s + k\right)(m-2)\frac{\Gamma(m-1)\Gamma(2s+k+1)}{\Gamma(2s+k+m-1)},\\
\gamma_{2s+1,k} &= \frac{s!(s+k+1)!\Gamma\left(\frac{m}{2}+s\right)}{\Gamma\left(\frac{m}{2}+2s+k\right)^2}\Gamma\left(\frac{m}{2}+s + k\right)(m-2)\frac{\Gamma(m-1)\Gamma(2s+k+2)}{\Gamma(2s+k+m)}.
\end{align*}

\end{proposition}
\begin{proof}
As the equality in Proposition \ref{zonal equality stiefel integral} must hold for each $\ux$ and $\uy$, it suffices to calculate $\gamma_{\alpha,k}$ for $\ux= \uy = \uom \in\mS^{m-1}$ and take the scalar part of the equation. The right-hand side of (\ref{integral over stiefel}) will now be equal to

\[
\gamma_{\alpha,k}\frac{\alpha+k+m-2}{m-2}C_{\alpha+k}^{m/2-1}(1).
\]
First we make the following observation for the left hand side of (\ref{integral over stiefel}). There exists $\sigma\in\mbox{Spin}(m)$ such that $\uta = \overline{\sigma}(e_1 + ie_2)\sigma$.\\
For ease of notation let us write $\uet = e_1 + ie_2$. Using Lemma \ref{Lemma spin on M[psi]} we get
\begin{equation}\label{integrand}
M[\psi_{\uta,\alpha,k}](\uom)M[\psi_{\uta,\alpha,k}](\uom)^{\dagger} =\overline{\sigma}M[\psi_{\uet,\alpha,k}](\sigma\uom\overline{\sigma})M[\psi_{\uet,\alpha,k}](\sigma\uom\overline{\sigma})^{\dagger} \sigma.
\end{equation}
The integral in (\ref{integral over stiefel}) can be interpreted as the mean of (\ref{integrand}) over the Stiefel manifold. But using the notation $\uta = \overline{\sigma}\uet\sigma$, we can also interpret it as a mean over the Spin-group $\mbox{Spin}(m)$ (see e.g. \cite{Szego}). Hence using Lemma \ref{Lemma spin on M[psi]} we get

\begin{align}\label{Stiefel naar Spin}
&\frac{1}{A_m A_{m-1}}\int_{\mS^{m-1}}\int_{\mS^{m-2}} M[\psi_{\uta,\alpha,k}](\uom)M[\psi_{\uta,\alpha,k}](\uom)^{\dagger} dS(\us)dS(\ut) \nonumber\\
&= \frac{1}{\mbox{Vol}(\mbox{Spin}(m))} \int_{\mbox{Spin}(m)} M[\psi_{\overline{\sigma}\uet\sigma,\alpha,k}](\uom)M[\psi_{\overline{\sigma}\uet\sigma,\alpha,k}](\uom)^{\dagger} dS(\sigma)\\
&= \frac{1}{\mbox{Vol}(\mbox{Spin}(m))} \int_{\mbox{Spin}(m)}\overline{\sigma}M[\psi_{\uet,\alpha,k}](\sigma\uom\overline{\sigma})M[\psi_{\uet,\alpha,k}](\sigma\uom\overline{\sigma})^{\dagger} \sigma dS(\sigma). \nonumber
\end{align}
Since the action of the Spin element on a vector is a rotation, it is easy to see that it is an automorphism on $k$-vectors, see \cite{Friederich spin, Gil Mur spin}. Hence we get that the scalar part of (\ref{Stiefel naar Spin}) is equal to

\begin{align*}
\frac{1}{\mbox{Vol}(\mbox{Spin}(m))} &\int_{\mbox{Spin}(m)} \left[M[\psi_{\uet,\alpha,k}](\sigma\uom\overline{\sigma})M[\psi_{\uet,\alpha,k}](\sigma\uom\overline{\sigma})^{\dagger}\right]_0  dS(\sigma)\\
&=\frac{1}{A_m} \int_{\mS^{m-1}} \left[M[\psi_{\uet,\alpha,k}](\uom)M[\psi_{\uet,\alpha,k}](\uom))^{\dagger}\right]_0  dS(\uom).
\end{align*}
Expanding our integrand yields
\[
M[\psi_{\uet,\alpha,k}](\uom)M[\psi_{\uet,\alpha,k}](\uom)^{\dagger} =  \sum_{j,l=0}^\alpha \mu_{j,\alpha,k}\mu_{l,\alpha,k} \uom^j \psi_{\uet,\alpha - j,k}(\uom)\psi_{\uet,\alpha - l,k}(\uom)^{\dagger} (\uom^{\dagger})^l.
\]

In order to continue, we need to know the parity of $\alpha-j$ and $\alpha-l$, so we will distinguish the cases where $\alpha$ is even and where $\alpha$ is odd. Suppose $\alpha = 2s$ even, then we have 

\begin{align*}
\sum_{j,l=0}^\alpha &\mu_{j,\alpha,k}\mu_{l,\alpha,k} \uom^j \psi_{\uet,\alpha - j,k}(\uom)\psi_{\uet,\alpha - l,k}(\uom)^{\dagger} (\uom^{\dagger})^l\\
=&\sum_{j=0}^{s}\sum_{l=0}^{s} \mu_{2j,\alpha,k}\mu_{2l,\alpha,k} (-1)^{j+l} \psi_{\uet,2s - 2j,k}(\uom)\psi_{\uet,2s - 2l,k}(\uom)^{\dagger} \\
&+\sum_{j=0}^{s-1}\sum_{l=0}^{s}\mu_{2j+1,\alpha,k}\mu_{2l,\alpha,k} (-1)^{j+l} \uom \psi_{\uet,2(s - j - 1) + 1,k}(\uom)\psi_{\uet,2s - 2l,k}(\uom)^{\dagger} \\
&+ \sum_{j=0}^{s}\sum_{l=0}^{s-1} \mu_{2j,\alpha,k}\mu_{2l+1,\alpha,k} (-1)^{j+l + 1} \psi_{\uet,2s - 2j,k}(\uom)\psi_{\uet,2s - 2l - 1,k}(\uom)^{\dagger} \uom\\
&+\sum_{j=0}^{s-1}\sum_{l=0}^{s-1}\mu_{2j+1,\alpha,k}\mu_{2l+1,\alpha,k} (-1)^{j + l + 1} \uom \psi_{\uet,2s - 2j - 1,k}(\uom)\psi_{\uet,2s - 2l - 1,k}(\uom)^{\dagger} \uom.
\end{align*}
where we used that $\uom^2 = -1$ and $\uom^{\dagger} = -\uom$ for $\uom\in\mS^{m-1}$. So now we have 4 different double sums we need to integrate. We can calculate the integrals using the previous lemmas and then we can take the scalar part in order to simplify the next calculations.

\begin{itemize}
\item
Rewriting the terms of the first sum, gives us

\begin{align*}
\mu_{2j, \alpha,k}\mu_{2l, \alpha,k} &(-1)^{j+l} \langle \uom, \uet\rangle^{s-j+k}\langle \uom, \uetd\rangle^{s-j}\uet\uetd \overline{\langle \uom, \uet\rangle}^{s-l+k}\overline{\langle \uom, \uetd\rangle}^{s-l}\\
&=\mu_{2j, \alpha,k}\mu_{2l, \alpha,k} (-1)^{j+l+k} \langle \uom, \uet\rangle^{2s-j-l+k}\langle \uom, \uetd\rangle^{2s-j-l+k}\uet\uetd
\end{align*}

Hence using Lemma \ref{lem4.7} we get

\begin{align*}
\frac{1}{A_m}&\int_{\mS^{m-1}} \mu_{2j, \alpha,k}\mu_{2l, \alpha,k} (-1)^{j+l} \psi_{\uet,2s - 2j,k}(\uom)\psi_{\uet,2s - 2l,k}(\uom)^{\dagger}dS(\uom)=\\
&=\frac{1}{A_m}\int_{\mS^{m-1}} \mu_{2j, \alpha,k}\mu_{2l, \alpha,k} (-1)^{j+l+k} \langle \uom, \uet\rangle^{2s-j-l+k}\langle \uom, \uetd\rangle^{2s-j-l+k}\uet\uetd dS(\uom)\\
&=\frac{1}{A_m}\mu_{2j, \alpha,k}\mu_{2l, \alpha,k} 2\pi^{\frac{m}{2}} \frac{\Gamma(2s-j-l+k+1)}{\Gamma\left(2s-j-l+k+\frac{m}{2}\right)} \uet\uetd 
\end{align*}
Hence the scalar part will be equal to

\[
\frac{1}{A_m}2\mu_{2j, \alpha,k}\mu_{2l, \alpha,k} 2\pi^{\frac{m}{2}} \frac{\Gamma(2s-j-l+k+1)}{\Gamma\left(2s-j-l+k+\frac{m}{2}\right)}.
\]

\item
For the terms of our second sum we get

\[
\mu_{2j+1, \alpha,k}\mu_{2l, \alpha,k} (-1)^{j+l} \uom \langle \uom, \uet\rangle^{s-j+k}\langle \uom, \uetd\rangle^{s-j-1} \uetd\uet\uetd \overline{\langle \uom, \uet\rangle}^{s-l+k}\overline{\langle \uom, \uetd\rangle}^{s-l}.
\]
Now using Remark \ref{remark uometd} we obtain

\begin{align*}
\frac{1}{A_m}&\int_{\mS^{m-1}} 4\mu_{2j+1, \alpha,k}\mu_{2l, \alpha,k} (-1)^{j+l+k} \langle \uom, \uet\rangle^{2s-j-l+k}\langle \uom, \uetd\rangle^{2s-j-l+k-1} \uom \uetd dS(\uom)\\
&=-\frac{1}{A_m}4\mu_{2j+1, \alpha,k}\mu_{2l, \alpha,k} \pi^{\frac{m}{2}} \frac{\Gamma(2s-j-l+k+1)}{\Gamma\left(2s-j-l+k+\frac{m}{2}\right)}\uet\uetd.
\end{align*}
 So now the scalar part yields
 
 \[
 -\frac{1}{A_m}8\mu_{2j+1, \alpha,k}\mu_{2l, \alpha,k} \pi^{\frac{m}{2}} \frac{\Gamma(2s-j-l+k+1)}{\Gamma\left(2s-j-l+k+\frac{m}{2}\right)}.
 \]

\item
For the terms of our third sum we get

\[
\mu_{2j, \alpha,k}\mu_{2l+1, \alpha,k} (-1)^{j+l+1}  \langle \uom, \uet\rangle^{s-j+k}\langle \uom, \uetd\rangle^{s-j} \uet\uetd\uet \overline{\langle \uom, \uet\rangle}^{s-l+k}\overline{\langle \uom, \uetd\rangle}^{s-l-1}\uom.
\]
Now using Lemma \ref{lem7.7} we get 

\begin{align*}
\frac{1}{A_m}&\int_{\mS^{m-1}} 4\mu_{2j, \alpha,k}\mu_{2l+1, \alpha,k} (-1)^{j+l+k} \langle \uom, \uet\rangle^{2s-j-l+k-1}\langle \uom, \uetd\rangle^{2s-j-l+k} \uet\uom  dS(\uom)\\
&=-\frac{1}{A_m}4\mu_{2j, \alpha,k}\mu_{2l+1, \alpha,k} \pi^{\frac{m}{2}} \frac{\Gamma(2s-j-l+k+1)}{\Gamma\left(2s-j-l+k+\frac{m}{2}\right)}\uet\uetd.
\end{align*}
Hence the scalar part becomes

 \[
 -\frac{1}{A_m}8\mu_{2j, \alpha,k}\mu_{2l+1, \alpha,k} \pi^{\frac{m}{2}} \frac{\Gamma(2s-j-l+k+1)}{\Gamma\left(2s-j-l+k+\frac{m}{2}\right)}.
 \]
 Note that this is equal to the terms of the second sum where we interchange $j$ and $l$.
\item
Lastly, for the terms of the fourth sum, we have

\[
\mu_{2j+1, \alpha,k}\mu_{2l+1, \alpha,k} (-1)^{j+l+1} \uom\langle \uom, \uet\rangle^{s-j+k}\langle \uom, \uetd\rangle^{s-j-1}\uet\uetd\uet\uetd \overline{\langle \uom, \uet\rangle}^{s-l+k}\overline{\langle \uom, \uetd\rangle}^{s-l -1}\uom.
\]

Using Lemma \ref{lemma integral xtaux} and taking the scalar part, the terms of this sum become

\[
\frac{1}{A_m}\mu_{2j+1, \alpha,k}\mu_{2l+1, \alpha,k} 16\pi^{\frac{m}{2}}\frac{(2s-j-l+k-1)!}{\Gamma(2s-j-l+k -1+\frac{m}{2})}.
\]
\end{itemize}
So now we have

\begin{align*}
\gamma_{2s,k}& \frac{2s+k+m-2}{m-2} C_{2s+k}^{\frac{m}{2}-1}(1)\\
=&\left[\frac{1}{A_m A_{m-1}}\int_{\mS^{m-1}}\int_{\mS^{m-2}} M[\psi_{\uta,\alpha,k}](\uom)M[\psi_{\uta,\alpha,k}](\uom)^{\dagger} dS(\us)dS(\ut)\right]_0\\
=&\frac{4}{A_m}\pi^{\frac{m}{2}}\left[\sum_{j=0}^s \sum_{l=0}^s \mu_{2j, \alpha,k}\mu_{2l, \alpha,k}  \frac{\Gamma(2s-j-l+k+1)}{\Gamma\left(2s-j-l+k+\frac{m}{2}\right)}\right.\\
&-4 \sum_{j=0}^{s-1} \sum_{l=0}^s\mu_{2j+1, \alpha,k}\mu_{2l, \alpha,k}  \frac{\Gamma(2s-j-l+k+1)}{\Gamma\left(2s-j-l+k+\frac{m}{2}\right)}\\
&+4\left.\sum_{j=0}^{s-1} \sum_{l=0}^{s-1}\mu_{2j+1, \alpha,k}\mu_{2l+1, \alpha,k} \frac{(2s-j-l+k-1)!}{\Gamma(2s-j-l+k -1+\frac{m}{2})}\right]\\
=&2\Gamma\left(\frac{m}{2}\right)\left[\sum_{j=0}^s \sum_{l=0}^s \mu_{2j, \alpha,k}\mu_{2l, \alpha,k}  \frac{\Gamma(2s-j-l+k+1)}{\Gamma\left(2s-j-l+k+\frac{m}{2}\right)}\right.\\
&-4 \sum_{j=0}^{s-1} \sum_{l=0}^s\mu_{2j+1, \alpha,k}\mu_{2l, \alpha,k}  \frac{\Gamma(2s-j-l+k+1)}{\Gamma\left(2s-j-l+k+\frac{m}{2}\right)}\\
&+4\left.\sum_{j=0}^{s-1} \sum_{l=0}^{s-1}\mu_{2j+1, \alpha,k}\mu_{2l+1, \alpha,k} \frac{(2s-j-l+k-1)!}{\Gamma(2s-j-l+k -1+\frac{m}{2})}\right].
\end{align*}

After doing the necessary calculations which are contained in Appendix \ref{appendix B}, we end up with the result stated in the Proposition.\\

The case $\alpha = 2s+1$ odd, is treated in a similar way.

\end{proof}

We can now find an inversion for the monogenic Hua-Radon transform. First of all, we define the dual Radon transform as follows:
\begin{definition}
The \textit{dual Radon transform} $\tilde{R}[F(\uz,\uta)]$ of a function $F(\uz,\uta)\in\oplus_{\alpha\in\mN} \mathfrak{M}^{j,\alpha}(\uta)$ is defined as

\[
\tilde{R}[F(\uz,\uta)] = \frac{1}{A_{m}A_{m-1}}\int_{\mS^{m-1}}\left(\int_{\mS^{m-2}}F(\uz,\uta) dS(\us)\right)dS(\ut), \qquad \uta = \ut+i\us
\]
where $\mS^{m-2}\subseteq \mS^{m-1}$ is the $(m-2)$-sphere orthogonal to $\ut$.
\end{definition}
Using this dual transform, we can construct an inversion. We get the following theorem.

\begin{theorem}\label{thm::dual of mono}
Let $M_{\ell}(z)$ be a monogenic polynomial of degree $\ell$. We have 

\[
\tilde{R}[\cM_{\uta, j}(\uz^{n} M_{\ell}(\uz))]  = 0 \qquad  \mathrm{if\ } n\neq j
\]
and for $f(z) = \uz^j M_l(z)$ we have

\[
\tilde{R}[\cM_{\uta, j}(\uz^j M_{\ell}(\uz))]  = \vartheta_{j,\ell} \uz^j M_{\ell}(z)
\]
with

\[
\vartheta_{j,\ell} = \frac{(\ell+1)!(m-2)!}{2(\ell+m-2)!}.
\]
\end{theorem}
\begin{proof}
We have

\begin{align*}
\tilde{R}[\cM_{\uta,j}(\uz^n M_{\ell})] &= \tilde{R}\left[\int_{\mS^{m-1}}\int_0^{\pi} K^j(\uz,e^{-i\theta}\uom) (e^{i\theta}\uom)^n M_{\ell}(e^{i\theta}\omega)dS(\uom)d\theta\right]\\
&=\tilde{R}\left[\int_{\mS^{m-1}}\int_0^{\pi} \sum_{\alpha\in\mN}\uz^j L^\alpha(z,e^{-i\theta}\omega) (e^{i\theta}\uom)^{-j}(e^{i\theta}\uom)^n M_{\ell}(e^{i\theta}\omega)dS(\uom)d\theta\right]\\
&=\int_{\mS^{m-1}}\int_0^{\pi} \sum_{\alpha\in\mN}\uz^j \tilde{R}\left[L^\alpha(z,e^{-i\theta}\omega)\right] (e^{i\theta}\uom)^{n-j} M_{\ell}(e^{i\theta}\omega)dS(\uom)d\theta.
\end{align*}

Using Proposition \ref{zonal equality stiefel integral}, we get

\begin{align*}
\tilde{R}\left[L^\alpha(z,e^{-i\theta}\omega)\right]&= \tilde{R}\left[\sum_{k=0}^\infty \lambda_k^{\alpha} M[\psi_{\uta,\alpha,k}](z)M[\psi_{\uta,\alpha,k}](e^{i\theta}\uom)^{\dagger}\right]\\
&=\sum_{k=0}^\infty \lambda_k^{\alpha} \tilde{R}\left[M[\psi_{\uta,\alpha,k}](z)M[\psi_{\uta,\alpha,k}](e^{i\theta}\uom)^{\dagger}\right]\\
&= \sum_{k=0}^\infty \lambda_k^{\alpha}\gamma_{\alpha,k}\cC_{\alpha+k}(\uz, e^{-i\theta }\uom).
\end{align*}
Hence

\begin{align}
\tilde{R}[\cM_{\uta,j}(\uz^j M_{\ell})]=&\uz^j\sum_{\alpha\in\mN}\sum_{k=0}^\infty\int_0^{\pi} e^{i\theta(n+\ell-j-\alpha-k)}d\theta \label{Dual of mono}\\
&\times\int_{\mS^{m-1}}  \lambda_k^{\alpha}\gamma_{\alpha,k}\cC_{\alpha+k}(\uz, \uom)\uom^{n-j}M_{\ell}(\omega)dS(\uom).\nonumber
\end{align}
Using Remark \ref{rem::sphere inprod 0 mono}, we get that (\ref{Dual of mono}) is non-zero if and only if $n=j$. If $n=j$, we can use the reproducing properties of the zonal spherical monogenics so that

\begin{align*}
\tilde{R}[\cM_{\uta,j}(\uz^j M_{\ell})] &= \uz^j\pi A_m \left(\sum_{2s+k=\ell} \lambda_k^{2s}\gamma_{2s,k} + \sum_{2s+1+k=\ell}\lambda_k^{2s+1}\gamma_{2s+1,k}\right) M_{\ell}(z)\\
=& \uz^j\pi A_m \left(\sum_{2s+k=\ell} \frac{1}{\pi A_m}\frac{(2s+k)!(m-2)!}{2(2s+k+m-2)!}\right.\\
& + \left.\sum_{2s+1+k=\ell}\frac{1}{\pi A_m}\frac{(2s+k+1)!(m-2)!}{2(2s+k+m-1)!}\right) M_{\ell}(z)\\
=&\uz^j(\ell+1) \frac{\ell!(m-2)!}{2(\ell+m-2)!}M_{\ell}(z)\\
=&\uz^j \frac{(\ell+1)!(m-2)!}{2(\ell+m-2)!}M_{\ell}(z).
\end{align*}
\end{proof}

Now using the gamma operator $\Gamma_{\uz} = -\uz\wedge\upz$ we get

\begin{align*}
\Gamma_{\uz} \left(\uz^{2j} M_{\ell}(\uz)\right) &= -\ell \uz^{2j} M_{\ell}(\uz),\\
\Gamma_{\uz} \left(\uz^{2j+1} M_{\ell}(\uz)\right) &= (\ell+m-1) \uz^{2j+1} M_{\ell}(\uz).
\end{align*}
Using this result we can write $\vartheta_{j,l}$ calculated in Theorem \ref{thm::dual of mono} as an operator:

\begin{align*}
\vartheta_{2j,\ell} &= \vartheta_{2j,-\Gamma_{\uz}},\\
\vartheta_{2j+1,\ell} &= \vartheta_{2j+1,\Gamma_{\uz}-m+1}.
\end{align*}
This yields an inversion formula for the monogenic Hua-Radon transform for functions of the type $\uz^j M_{\ell}(\uz)$ given in the following Theorem:

\begin{theorem}\label{rem::total monogenic}
For any $j,\ell\in \mN$, we have
\begin{align*}
\uz^{2j} M_{\ell}(\uz) &= \vartheta_{2j,-\Gamma_{\uz}}^{-1}\tilde{R}[\cM_{\uta, 2j}(\uz^{2j} M_{\ell}(\uz))],\\
\uz^{2j+1} M_{\ell}(\uz) &= \vartheta_{2j+1,\Gamma_{\uz}-m+1}^{-1}\tilde{R}[\cM_{\uta, 2j+1}(\uz^{2j+1} M_{\ell}(\uz))].
\end{align*}

If we were to consider a \emph{total monogenic Hua-Radon transform} defined as the direct sum

\[
\oplus_{j=0}^{\infty} \cM_{\uta,j}: \mathcal{OL}^2(LB(0,1) \to \bigoplus_{j=0}^{\infty}\bigoplus_{\alpha\in\mN} \mathfrak{M}^{j,\alpha}(\uta):f\mapsto \sum_{j=0}^{\infty} \cM_{\uta,j}(f)
\]
we can define a inversion operator for all $\mathcal{OL}^2(LB(0,1))$ as

\[
\bigoplus_{j=0}^{\infty}\left(\vartheta_{2j,-\Gamma_{\uz}}^{-1}\tilde{R}[\cM_{\uta, 2j}(\cdot)] + \vartheta_{2j+1,\Gamma_{\uz}-m+1}^{-1}\tilde{R}[\cM_{\uta, 2j+1}(\cdot)]\right).
\]
\end{theorem}

\section{Conclusions}
\setcounter{equation}{0}

In Theorem \ref{Thm coeff mono hua} we completed the computations of the kernel of the monogenic Hua-Radon transform. Following the techniques used to invert the Szeg\H o-Radon transform in \cite{Szego}, we obtained a projection operator mapping a holomorphic function onto the term of the form $\uz^j M_k(\uz)$ in its Fischer decomposition. This process leads to a full inversion formula for the total monogenic Hua-Radon transform defined in Theorem \ref{rem::total monogenic}.

\newpage
\appendix
\section{Computation of $\lambda_{k}^{\alpha}$ from Theorem \ref{Thm coeff mono hua} }\label{appendix A}
\setcounter{equation}{0}
In order to determine the coefficients $\lambda_k^{\alpha}$ from Proposition \ref{Thm coeff mono hua} as a closed formula, we first look at the following sums:

\begin{align*}
\sum_{l=0}^s  (-1)^l\frac{\Gamma(2s+k-l+1)}{\Gamma(l+1)\Gamma(s-l+1)\Gamma(s+k-l+1)} &= \sum_{l=0}^s  (-1)^l\frac{(2s+k-l)!}{l!(s-l)!(s+k-l)!}\frac{s!}{s!}\\
&= \sum_{l=0}^{s} (-1)^l \left(\begin{matrix}
s\\
l
\end{matrix}\right)\left(\begin{matrix}
2s+k-l\\
s+k-l
\end{matrix}\right),\\
\sum_{l=0}^s  (-1)^l\frac{\Gamma(2s+k-l+2)}{\Gamma(l+1)\Gamma(s-l+1)\Gamma(s+k-l+2)} &= \sum_{l=0}^s  (-1)^l\frac{(2s+k-l+1)!}{l!(s-l)!(s+k-l+1)!}\frac{s!}{s!}\\
&= \sum_{l=0}^{s} (-1)^l \left(\begin{matrix}
s\\
l
\end{matrix}\right)\left(\begin{matrix}
2s+k-l+1\\
s+k-l+1
\end{matrix}\right).
\end{align*}
Note that the last sum is just the first sum by shifting $k$ by $1$. This leads to the following:

\begin{proposition}\label{prop pascal}
For each $k$ we have

\begin{align}\label{first sum}
\sum_{l=0}^{s} (-1)^l \left(\begin{matrix}
s\\
l
\end{matrix}\right)\left(\begin{matrix}
2s+k-l\\
s+k-l
\end{matrix}\right) &= 1.
\end{align}
\end{proposition}
\begin{proof}
First of all we recall Pascal's formula:

\[
\left(\begin{matrix}
n\\
j
\end{matrix}\right) = \left(\begin{matrix}
n-1\\
j
\end{matrix}\right)+\left(\begin{matrix}
n-1\\
j-1
\end{matrix}\right)
\]
Hence for the sum in (\ref{first sum}) we get

\begin{align}
\sum_{l=0}^{s} (-1)^l \left(\begin{matrix}
s\\
l
\end{matrix}\right)\left(\begin{matrix}
2s+k-l\\
s+k-l
\end{matrix}\right) =& \left(\begin{matrix}
s\\
0
\end{matrix}\right)\left(\begin{matrix}
2s+k\\
s+k
\end{matrix}\right)  + (-1)^s\left(\begin{matrix}
s\\
s
\end{matrix}\right)\left(\begin{matrix}
s+k\\
k
\end{matrix}\right)\nonumber\\
&+ \sum_{l=1}^{s-1} (-1)^l \left(\left(\begin{matrix}
s-1\\
l
\end{matrix}\right)
+\left(\begin{matrix}
s-1\\
l-1
\end{matrix}\right)\right)\left(\begin{matrix}
2s+k-l\\
s+k-l
\end{matrix}\right)\nonumber\\
=& \left(\begin{matrix}
s-1\\
0
\end{matrix}\right)\left(\begin{matrix}
2s+k\\
s+k
\end{matrix}\right)  - (-1)^{s-1} \left(\begin{matrix}
s-1\\
s-1
\end{matrix}\right)\left(\begin{matrix}
s+k\\
k
\end{matrix}\right)\nonumber\\
&+ \sum_{l=1}^{s-1} (-1)^l \left(\left(\begin{matrix}
s-1\\
l
\end{matrix}\right)
+\left(\begin{matrix}
s-1\\
l-1
\end{matrix}\right)\right)\left(\begin{matrix}
2s+k-l\\
s+k-l
\end{matrix}\right)\nonumber\\
=&\sum_{l=0}^{s-1} (-1)^l \left(\begin{matrix}
s-1\\
l
\end{matrix}\right)\left(
\left(\begin{matrix}
2s+k-l\\
s+k-l
\end{matrix}\right)-\left(\begin{matrix}
2s+k-l-1\\
s+k-l-1
\end{matrix}\right)\right)\nonumber\\
=& \sum_{l=0}^{s-1} (-1)^l \left(\begin{matrix}
s-1\\
l
\end{matrix}\right)\left(\begin{matrix}
2s+k-l-1\\
s+k-l
\end{matrix}\right) \label{first sum.2}.
\end{align}
Hence we see that (\ref{first sum.2}) is equal to sum (\ref{first sum}) under the transform $s\mapsto s-1$ and $k\mapsto k+1$. Hence repeating this process allows us to lower the upper index of the sum, so that we end up with
\begin{align*}
\sum_{l=0}^{s} (-1)^l \left(\begin{matrix}
s\\
l
\end{matrix}\right)\left(\begin{matrix}
2s+k-l\\
s+k-l
\end{matrix}\right)=& \sum_{l=0}^{s-s} (-1)^l \left(\begin{matrix}
s-s\\
l
\end{matrix}\right)\left(\begin{matrix}
2s+k-l-s\\
s+k-l
\end{matrix}\right)\\
=&\left(\begin{matrix}
0\\
0
\end{matrix}\right)\left(\begin{matrix}
s+k-0\\
s+k-0
\end{matrix}\right)\\
=& 1.
\end{align*}
\end{proof}

Hence by Proposition \ref{prop pascal} we have

\begin{align*}
(\lambda_k^{2s})^{-1} =&\frac{8\pi^{\frac{m}{2}+1}s!(s+k)!}{\Gamma(\frac{m}{2}+2s+k)}\left( 1 - \sum_{l=0}^{s-1}  (-1)^l \frac{s}{\frac{m}{2}+2s+k-l-1} \left(\begin{matrix}
s-1\\
l
\end{matrix}\right)\left(\begin{matrix}
2s+k-l\\
s+k-l
\end{matrix}\right)\right),\\
(\lambda_k^{2s+1})^{-1} =&\frac{32\pi^{\frac{m}{2}+1}s!(s+k+1)!}{\Gamma(\frac{m}{2}+2s+k+1)}\left( 1 - \sum_{l=0}^{s}  (-1)^l \frac{s+1}{\frac{m}{2}+2s+k-l} \left(\begin{matrix}
s\\
l
\end{matrix}\right)\left(\begin{matrix}
2s+k-l+1\\
s+k-l
\end{matrix}\right)\right).
\end{align*}
We see that for the remaining sums, if we shift $s$ by 1 and $k$ by $-1$ in the first sum, we get the second sum. Hence it suffices to calculate the first sum.

\begin{proposition}\label{prop Roy}
We have
\begin{align*}
\sum_{l=0}^{s-1}  (-1)^l \frac{s}{\frac{m}{2}+2s+k-l-1} \left(\begin{matrix}
s-1\\
l
\end{matrix}\right)\left(\begin{matrix}
2s+k-l\\
s+k-l
\end{matrix}\right) &= 1-\frac{\Gamma(s-1+\frac{m}{2})\Gamma(\frac{m}{2}+s+k)}{\Gamma(\frac{m}{2}-1)\Gamma(\frac{m}{2}+2s+k)}\\
&=1-\frac{(\frac{m}{2}-1)_s}{(\frac{m}{2}+s+k)_s}.
\end{align*}
where $(a)_n=a(a+1)\ldots (a+n-1)$ denotes the Pochhammer symbol.
\end{proposition}
\begin{proof}
First of all, we see that we can rewrite this series as a multiple of a hypergeometric function ${}_3F_2$:

\begin{align*}
\sum_{l=0}^{s-1}  (-1)^l &\frac{s}{\frac{m}{2}+2s+k-l-1} \left(\begin{matrix}
s-1\\
l
\end{matrix}\right)\left(\begin{matrix}
2s+k-l\\
s+k-l
\end{matrix}\right)\\
=& \frac{s\Gamma(2s+k+1)}{\left(\frac{m}{2}+2s+k-1\right)\Gamma(s+k+1)\Gamma(s+1)}\\
&\times\underbrace{{}_3F_2([1-s,-k-s,-\frac{m}{2}-2s-k+1],[-k-2s,-\frac{m}{2}-2s-k+2];1)}_{:=\phi}.
\end{align*}
In order to determine the value of this hypergeometric function, we will use the following result from \cite[p873, appendix, equation (II)]{hypergeom}:

\begin{align}\label{Hypergeom result}
{}_3F_2([a,b,-N],[d,e]);1) =& (-1)^N \frac{(1+a-d-N-e+b)_N}{(d)_N}\\
&\times{}_3F_2([e-a,e-b,-N],[e-b-a+d,e]);1).\nonumber
\end{align}
If we now apply this to our hypergeometric series with

\begin{align*}
N &= s-1 & b &= -k-s & a = -\frac{m}{2}-2s-k+1\\
d &= -\frac{m}{2}-2s-k+2 = a+1 & e &= -k-2s = b-N -1,
\end{align*}
we get that 

\begin{align*}
\phi &= (-1)^N \frac{(1)_N}{(a+1)_N}{}_3F_2([b-N-a-1,-N-1,-N],[-N,b-N-1];1)\\
&= \frac{(-1)^N(1)_N}{(a+1)_N} \sum_{j=0}^{N} \frac{(b-N-a-1)_j(-N-1)_j(-N)_j}{(-N)_j(b-N-1)_j j!}\\
&= \frac{(-1)^N(1)_N}{(a+1)_N} \left(\sum_{j=0}^{N+1} \frac{(b-N-a-1)_j(-N-1)_j}{(b-N-1)_j j!} - (-1)^{N+1}\frac{(b-N-a-1)_{N+1}}{(b-N-1)_{N+1}}\right)\\
&= \frac{(-1)^N N!}{(a+1)_N} \left({}_2F_1([-N-1,b-c-N-1],[b-N-1];1) - (-1)^{N+1}\frac{(b-N-a-1)_{N+1}}{(b-N-1)_{N+1}}\right).
\end{align*}
We can now use Gauss hypergeometric theorem \cite[Chapter 15]{NIST} so that we end up with

\begin{align*}
\phi &= \frac{(-1)^N N!}{(a+1)_N} \left(\frac{(a)_{N+1}}{(b-N-1)_{N+1}} - (-1)^{N+1}\frac{(b-N-a-1)_{N+1}}{(b-N-1)_{N+1}}\right)\\
&=\frac{(-1)^N N!}{(b-N-1)_{N+1}} \left(\frac{(a)_{N+1}}{(a+1)_N} - \frac{(1+a-b)_{N+1}}{(a+1)_N} \right)\\
&=\frac{-N!}{(1-b)_{N+1}} \left(a - \frac{(1+a-b)_{N+1}}{(a+1)_N} \right)\\
&=\frac{(s-1)!}{(k+s+1)_{s}} \left(\frac{m}{2}+2s+k-1 + \frac{\left(-\frac{m}{2}-s+2\right)_{s}}{\left(-\frac{m}{2}-2s-k+2\right)_{s-1}} \right)\\
&=\frac{\Gamma(s)\Gamma(s+k+1)}{\Gamma(2s+k+1)} \left(\frac{m}{2}+2s+k-1 - \frac{\left(\frac{m}{2}-1\right)_{s}}{\left(\frac{m}{2}+s+k\right)_{s-1}} \right).
\end{align*}
Thus now our sum becomes

\begin{align*}
\sum_{l=0}^{s-1}  (-1)^l &\frac{s}{\frac{m}{2}+2s+k-l-1} \left(\begin{matrix}
s-1\\
l
\end{matrix}\right)\left(\begin{matrix}
2s+k-l\\
s+k-l
\end{matrix}\right)\\
=& \frac{1}{\left(\frac{m}{2}+2s+k-1\right)}\left(\frac{m}{2}+2s+k-1 - \frac{\left(\frac{m}{2}-1\right)_{s}}{\left(\frac{m}{2}+s+k\right)_{s-1}} \right)\\
=& 1-\frac{(\frac{m}{2}-1)_s}{(\frac{m}{2}+s+k)_s}\\
=& 1-\frac{\Gamma(s-1+\frac{m}{2})\Gamma(\frac{m}{2}+s+k)}{\Gamma(\frac{m}{2}-1)\Gamma(\frac{m}{2}+2s+k)}.
\end{align*}

\end{proof}

\newpage
\section{Computation of the constants $\gamma_{\alpha,k}$ from Proposition \ref{prop comp gam}}\label{appendix B}

Plugging in the explicit expressions for $\mu_{j,\alpha,l}$ (see Lemma \ref{Lemma Fisch decomp M[psi]}), the summations at the end of the proof of Proposition \ref{prop comp gam} become

\begin{align}
\sum_{j=0}^{s}\sum_{l=0}^{s}\mu_{2j}&\mu_{2l} \frac{\Gamma(2s-j-l+k+1)}{\Gamma\left(2s-j-l+k+\frac{m}{2}\right)}\nonumber\\
&=\frac{(s!)^2 ((s+k)!)^2}{\Gamma(\frac{m}{2}+2s+k)^2}\sum_{j=0}^{s} (-1)^{j}\frac{\Gamma(\frac{m}{2}+2s+k-j)}{j!(s-j)!(s+k-j)!}\label{even som 1.1}\\
&\times \sum_{l=0}^{s} (-1)^{l}\frac{\Gamma(\frac{m}{2}+2s+k-l)}{l!(s-l)!(s+k-l)!}\frac{\Gamma(2s-j-l+k+1)}{\Gamma\left(2s-j-l+k+\frac{m}{2}\right)},\label{even som 1.2}\\
\sum_{j=0}^{s-1}\sum_{l=0}^{s}\mu_{2j+1}&\mu_{2l} \frac{\Gamma(2s-j-l+k+1)}{\Gamma\left(2s-j-l+k+\frac{m}{2}\right)}\nonumber\\
&=\frac{(s!)^2 ((s+k)!)^2}{\Gamma(\frac{m}{2}+2s+k)^2}\sum_{j=0}^{s-1} (-1)^{j}\frac{\Gamma(\frac{m}{2}+2s+k-j-1)}{2j!(s-j-1)!(s+k-j)!}\label{even som 2.1}\\
&\times \sum_{l=0}^{s} \frac{(-1)^{l}\Gamma(\frac{m}{2}+2s+k-l)\Gamma(2s-j-l+k+1)}{l!(s-l)!(s+k-l)!\Gamma\left(2s-j-l+k+\frac{m}{2}\right)},\label{even som 2.2}\\
\sum_{j=0}^{s-1}\sum_{l=0}^{s-1}\mu_{2j+1}&\mu_{2l+1} \frac{\Gamma(2s-j-l+k+1)}{\Gamma\left(2s-j-l+k+\frac{m}{2}\right)}\nonumber\\
&=\frac{(s!)^2 ((s+k)!)^2}{\Gamma(\frac{m}{2}+2s+k)^2}\sum_{j=0}^{s-1} (-1)^{j}\frac{\Gamma(\frac{m}{2}+2s+k-j-1)}{2j!(s-j-1)!(s+k-j)!}\label{even som 3.1}\\
&\times \sum_{l=0}^{s-1} \frac{(-1)^{l}\Gamma(\frac{m}{2}+2s+k-l-1)\Gamma(2s-j-l+k)}{2l!(s-l-1)!(s+k-l)!\Gamma\left(2s-j-l+k-1+\frac{m}{2}\right)}.\label{even som 3.2}
\end{align}

Now note that the right-hand sides of (\ref{even som 1.2}) and (\ref{even som 2.2}) are identical and if we replace $s$ by $s-1$ and $k$ by $k+1$, the right-hand side of (\ref{even som 1.2}) transforms into (\ref{even som 3.2}). Moreover, if we write out the sums in the odd case, the summations with respect to $l$ are equal to (\ref{even som 1.2}) (if necessary change $k$ to $k+1$). Moreover (\ref{even som 2.1}) and (\ref{even som 3.1}) can be achieved by changing $s$ to $s-1$ and $k$ to $k+1$ in (\ref{even som 1.1}) and the same can be done for the remaining sums in the odd case. Hence it suffices to only calculate sum (\ref{even som 1.2}) and (\ref{even som 1.1}). We have the following:

\begin{proposition}\label{even som 2}
\begin{align*}
\sum_{l=0}^{s} (-1)^{l}\frac{\Gamma(\frac{m}{2}+2s+k-l)}{l!(s-l)!(s+k-l)!}\frac{\Gamma(2s-j-l+k+1)}{\Gamma\left(2s-j-l+k+\frac{m}{2}\right)} = 1.
\end{align*}
\end{proposition}
\begin{proof}
We have

\begin{align*}
\sum_{l=0}^{s} (-1)^{l}\frac{\Gamma(\frac{m}{2}+2s+k-l)}{l!(s-l)!(s+k-l)!}&\frac{\Gamma(2s-j-l+k+1)}{\Gamma\left(2s-j-l+k+\frac{m}{2}\right)} =\\
&\frac{\Gamma(\frac{m}{2}+2s+k)}{s!(s+k)!}\frac{\Gamma(2s-j+k+1)}{\Gamma\left(\frac{m}{2}+2s+k-j\right)} \phi,
\end{align*}
where

\[
\phi = {}_3F_2\left([a,b,-N], [d, e];1\right)
\]
with

\begin{align*}
a &= -s-k & b &= j-k-2s-\frac{m}{2}+1 & N &= s\\
d &= -2s+j-k = a-N+j & e &= -\frac{m}{2}-2s-k+1 = b-j.
\end{align*}
Using (\ref{Hypergeom result}), we have

\begin{align*}
\phi &= (-1)^N \frac{(1+a+b-d-e-N)_N}{(d)_N}{}_3F_2([e-a,e-b,-N],[e-a-b+d,e];1)\\
&= (-1)^N \frac{(1)_N}{(d)_N}{}_3F_2([e-a,-j,-N],[-N,e];1)\\
&= \frac{(-1)^N N!}{(d)_N} \sum_{l=0}^{j} \frac{(e-a)_l(-j)_l(-N)_l}{(-N)_l(e)_l l!}\\
&= \frac{\Gamma(-d+N+1) N!}{\Gamma(-d+1)} \sum_{l=0}^{j} \frac{(e-a)_l(-j)_l}{(e)_l l!}\\
&= \frac{\Gamma(-d+N+1) N!}{\Gamma(-d+1)} {}_2F_1([e-a, -j],[e];1)
\end{align*}
where we used that $j\leq s=N$ and hence the sum will terminate at $j$ instead of $N$. Now using Chu-Vandermonde identity (see \cite{NIST}) we have

\begin{align*}
{}_2F_1([e-a, -j],[e];1)&=\frac{(a)_j}{(e)_j}\\ &= \frac{(-s-k)_j}{(-\frac{m}{2}-2s-k+1)_j}\\
&= \frac{a(a+1)\ldots(a+j-1)}{e(e+1)\ldots(e+j-1)}\\
&= \frac{(-a)(-a-1)\ldots(-a-j+1)}{((-e)(-e-1)\ldots(-e-j+1)}\\
&=\frac{\Gamma(-a+1)\Gamma(-e-j+1)}{\Gamma(-a-j+1)\Gamma(-e+1)}
\end{align*}
as $-a-j+1\geq 0$, $-a\geq 0$, $-e\geq 0$ and $-e-j+1\geq 0$.
Thus for our total sum, we have

\begin{align*}
\sum_{l=0}^{s} (-1)^{l}&\frac{\Gamma(\frac{m}{2}+2s+k-l)}{l!(s-l)!(s+k-l)!}\frac{\Gamma(2s-j-l+k+1)}{\Gamma\left(2s-j-l+k+\frac{m}{2}\right)} \\
=&\frac{\Gamma(\frac{m}{2}+2s+k)}{s!(s+k)!}\frac{\Gamma(2s-j+k+1)}{\Gamma\left(\frac{m}{2}+2s+k-j\right)}\frac{\Gamma(-d+N+1) N!}{\Gamma(-d+1)}\frac{\Gamma(-a+1)\Gamma(-e-j+1)}{\Gamma(-a-j+1)\Gamma(-e+1)}\\
=&\frac{\Gamma(-e+1)}{N!\Gamma(-a+1)} \frac{\Gamma(-d+1)}{\Gamma(-e-j+1)}\frac{\Gamma(-d+N+1) N!}{\Gamma(-d+1)}\frac{\Gamma(-a+1)\Gamma(-e-j+1)}{\Gamma(-a-j+1)\Gamma(-e+1)} \\
=& \frac{\Gamma(-d-N+1)}{\Gamma(-a-j+1)} \\
=& \frac{ \Gamma(-d-N+1)}{\Gamma(-d-N+1)}\\
=& 1.
\end{align*}

\end{proof}

\begin{proposition}\label{even som 1}
The right-hand side of (\ref{even som 1.1}) is equal to
\[
\sum_{j=0}^s (-1)^j\frac{\Gamma\left(\frac{m}{2} + 2s + k - j\right)}{j!\Gamma(s - j + 1)\Gamma(s + k - j + 1)} = \frac{\Gamma\left(\frac{m}{2}+s\right)\Gamma\left(\frac{m}{2}+s + k\right)}{s!(s + k)!\Gamma\left(\frac{m}{2}\right)}.
\]
\end{proposition}
\begin{proof}
Using Chu-Vandermonde identity (see \cite{NIST}), we have
\begin{align*}
\sum_{j=0}^s (-1)^j\frac{\Gamma\left(\frac{m}{2} + 2s + k - j\right)}{j!\Gamma(s - j + 1)\Gamma(s + k - j + 1)} =& \frac{\Gamma\left(\frac{m}{2} + 2s + k\right)}{\Gamma(s + 1)\Gamma(s + k + 1)}\\
&\times{}_2F_1\left([-s,-s-k],\left[-\frac{m}{2}-2s-k+1\right];1\right)\\
=&\frac{\Gamma\left(\frac{m}{2} + 2s + k\right)}{\Gamma(s + 1)\Gamma(s + k + 1)}\frac{\left(-\frac{m}{2}-s+1\right)_s}{\left(-\frac{m}{2}-2s-k+1\right)_s}\\
=& \frac{\Gamma\left(\frac{m}{2}+s\right)\Gamma\left(\frac{m}{2}+s + k\right)}{\Gamma(s + 1)\Gamma(s + k + 1)\Gamma\left(\frac{m}{2}\right)}.
\end{align*}
\end{proof}
Now combining the results of Proposition \ref{even som 2} and Proposition \ref{even som 1} we get the following:

\begin{proposition}
The constants $\gamma_{\alpha,k}$ are given by
\begin{align*}
\gamma_{2s,k} &= \frac{s!(s+k)!\Gamma\left(\frac{m}{2}+s-1\right)}{\Gamma\left(\frac{m}{2}+2s+k\right)^2}\Gamma\left(\frac{m}{2}+s + k\right)(m-2)\frac{\Gamma(m-1)\Gamma(2s+k+1)}{\Gamma(2s+k+m-1)},\\
\gamma_{2s+1,k} &= 4\frac{s!(s+k+1)!\Gamma\left(\frac{m}{2}+s-1\right)}{\Gamma\left(\frac{m}{2}+2s+k\right)^2}\Gamma\left(\frac{m}{2}+s + k\right)(m-2)\frac{\Gamma(m-1)\Gamma(2s+k+1)}{\Gamma(2s+k+m-1)}.
\end{align*}
\end{proposition}
\begin{proof}
Combining Proposition \ref{even som 2} and Proposition \ref{even som 1} we get

\begin{align*}
\gamma_{2s,k}\frac{2s+k+m-2}{m-2} C_{2s+k}^{\frac{m}{2}-1}(1) =&2\Gamma\left(\frac{m}{2}\right)\frac{s!^2(s+k)!^2}{\Gamma\left(\frac{m}{2}+2s+k\right)^2}\left[\frac{\Gamma\left(\frac{m}{2}+s\right)\Gamma\left(\frac{m}{2}+s + k\right)}{\Gamma(s + 1)\Gamma(s + k + 1)\Gamma\left(\frac{m}{2}\right)}\right.\\
&-4\frac{\Gamma\left(\frac{m}{2}+s-1\right)\Gamma\left(\frac{m}{2}+s + k\right)}{2\Gamma(s)\Gamma(s + k + 1)\Gamma\left(\frac{m}{2}\right)}\\
&+\left.4\frac{\Gamma\left(\frac{m}{2}+s-1\right)\Gamma\left(\frac{m}{2}+s + k\right)}{4\Gamma(s )\Gamma(s + k + 1)\Gamma\left(\frac{m}{2}\right)}\right]\\
=& 2\Gamma\left(\frac{m}{2}\right)\frac{s!^2(s+k)!^2}{\Gamma\left(\frac{m}{2}+2s+k\right)^2}\left[\frac{\Gamma\left(\frac{m}{2}+s\right)\Gamma\left(\frac{m}{2}+s + k\right)}{\Gamma(s + 1)\Gamma(s + k + 1)\Gamma\left(\frac{m}{2}\right)}\right.\\
&\left.-\frac{\Gamma\left(\frac{m}{2}+s-1\right)\Gamma\left(\frac{m}{2}+s + k\right)}{\Gamma(s)\Gamma(s + k + 1)\Gamma\left(\frac{m}{2}\right)}\right]\\
=& 2\Gamma\left(\frac{m}{2}\right)\frac{s!^2(s+k)!^2}{\Gamma\left(\frac{m}{2}+2s+k\right)^2}\frac{\Gamma\left(\frac{m}{2}+s-1\right)\Gamma\left(\frac{m}{2}+s + k\right)}{\Gamma(s)\Gamma(s + k + 1)\Gamma\left(\frac{m}{2}\right)}\\
&\times\left[\frac{\frac{m}{2}+s-1}{s}-1\right]\\
=&\frac{s!(s+k)!}{\Gamma\left(\frac{m}{2}+2s+k\right)^2}\Gamma\left(\frac{m}{2}+s-1\right)\Gamma\left(\frac{m}{2}+s + k\right)(m-2)
\end{align*}
and thus

\begin{align*}
\gamma_{2s,k} &= \frac{s!(s+k)!}{\Gamma\left(\frac{m}{2}+2s+k\right)^2}\Gamma\left(\frac{m}{2}+s-1\right)\Gamma\left(\frac{m}{2}+s + k\right)(m-2) \frac{m-2}{(2s+k+m-2)C_{2s+k}^{\frac{m}{2}-1}(1)}\\
&= \frac{s!(s+k)!\Gamma\left(\frac{m}{2}+s-1\right)}{\Gamma\left(\frac{m}{2}+2s+k\right)^2}\Gamma\left(\frac{m}{2}+s + k\right)(m-2) \frac{m-2}{2s+k+m-2}\frac{\Gamma(m-2)\Gamma(2s+k+1)}{\Gamma(2s+k+m-2)}\\
&= \frac{s!(s+k)!\Gamma\left(\frac{m}{2}+s-1\right)}{\Gamma\left(\frac{m}{2}+2s+k\right)^2}\Gamma\left(\frac{m}{2}+s + k\right)(m-2)\frac{\Gamma(m-1)\Gamma(2s+k+1)}{\Gamma(2s+k+m-1)}.
\end{align*}

Analogously we get for $\alpha = 2s+1$:

\begin{align*}
\gamma_{2s+1,k} &= \frac{s!(s+k+1)!}{\Gamma\left(\frac{m}{2}+2s+k+1\right)^2}\Gamma\left(\frac{m}{2}+s\right)\Gamma\left(\frac{m}{2}+s + k\right)(m-2) \frac{m-2}{(2s+k+m-1)C_{2s+k}^{\frac{m}{2}-1}(1)}\\
&= \frac{s!(s+k+1)!\Gamma\left(\frac{m}{2}+s\right)}{\Gamma\left(\frac{m}{2}+2s+k\right)^2}\Gamma\left(\frac{m}{2}+s + k\right)(m-2) \frac{m-2}{2s+k+m-1}\frac{\Gamma(m-2)\Gamma(2s+k+2)}{\Gamma(2s+k+m-1)}\\
&= \frac{s!(s+k+1)!\Gamma\left(\frac{m}{2}+s\right)}{\Gamma\left(\frac{m}{2}+2s+k\right)^2}\Gamma\left(\frac{m}{2}+s + k\right)(m-2)\frac{\Gamma(m-1)\Gamma(2s+k+2)}{\Gamma(2s+k+m)}.
\end{align*}

\end{proof}

 \section*{Acknowledgements}
The authors would like to thank Roy Oste for his advice on hypergeometric series, which was especially useful in establishing Proposition \ref{prop Roy}. The work of HDB is supported by the Research Foundation Flanders (FWO) under Grant EOS 30889451. 

%%%%%%%%%%%%%%%%%%%%%%%%%%%%%%%%%%%%%%%%%%%%%%%%%%%%%%%%%%%%%%%%%%%%%%%%%%%%%%%%%%%%%%%%%%%%%%%%%%%%%%%%%

\end{document}